\documentclass{amsart}
\usepackage[latin1]{inputenc}
\usepackage{hyperref}
\usepackage{texdraw}
\usepackage{amssymb}
\usepackage{commath}
\usepackage{tikz}
\usepackage{tikz-cd}
\usepackage{bbm}

\makeatletter

\newfont{\gothic}{eufm10 scaled 1100}

\theoremstyle{plain}
\newtheorem{thm}{Theorem}[section]
\numberwithin{equation}{section} 
\numberwithin{figure}{section} 
\theoremstyle{plain}
\newtheorem{cor}[thm]{Corollary} 
\theoremstyle{plain}
\theoremstyle{plain}
\theoremstyle{plain}
\newtheorem{lem}[thm]{Lemma} 
\theoremstyle{plain}
\newtheorem{prop}[thm]{Proposition} 
\theoremstyle{definition}
\newtheorem{exm}[thm]{Example} 
\theoremstyle{definition}
\newtheorem{const}[thm]{Construction} 
\theoremstyle{definition}
\newtheorem{Def}[thm]{Definition} 
\theoremstyle{remark}
\newtheorem{rem}[thm]{Remark}
\theoremstyle{remark}

\usepackage[arrow,matrix,curve,ps]{xy}
\xyoption{dvips}
\CompileMatrices

\makeatother

\begin{document}

\title{Crystallographic Tilings}

\date{\today}

\author{Hawazin Alzahrani and Thomas Eckl}




\address{Hawazin Alzahrani and Thomas Eckl, Department of Mathematical Sciences, The University of Liverpool, Mathematical
               Sciences Building, Liverpool, L69 7ZL, England, U.K.}

\email{thomas.eckl@liv.ac.uk}

\urladdr{http://pcwww.liv.ac.uk/~eckl/}

\maketitle

\begin{abstract}
Crystallographic tilings of the Euclidean space $\mathbb{E}^n$ are defined as simple tilings whose group of isometric automorphisms is crystallographic. To classify crystallographic tilings by their automorphism groups it is necessary to extend the standard equivalence relation of mutual local derivability to a version taking more general isometries than translations into account. This also requires the extension of the standard metrics on tiling spaces. Finally, a tiling with a given crystallographic group as automorphism group is constructed.
\end{abstract}


\pagestyle{myheadings}
\markboth{H.~ALZAHRANI AND T.~ECKL}{CRYSTALLOGRAPHIC TILINGS}

\setcounter{section}{-1}

\section{Introduction}

\noindent The study of crystallographic groups and tilings was intimately intertwined since the very beginnings of the subject (see \cite[Ch.1]{Sen95} for a nice account of the history before the 20th century). However, since the early constructions of aperiodic tilings by Berger \cite{Ber66}, Penrose \cite{Pen79} and others, the startling discovery of quasicrystals by Schechtman \cite{She84} in nature and the subsequent development of mathematical tools, borrowing a lot from ergodic theory and the theory of dynamical systems, to construct (aperiodic) tilings, to describe their properties and to draw conclusions on quasicrystals (see Sadun's introduction to \cite{Sad08} for a short but concise outline) not much work seemed to be devoted to investigations on the connections between crystallographic groups and tilings, systematically and in the light of the newly developed mathematical frameworks.

\noindent To start with, the authors were not able to locate anywhere in the literature the following natural definition:
\begin{Def}[$=$ Definition~\ref{cryst-tile-def}]
A crystallographic tiling of the Euclidean space $\mathbb{E}^n$ is a (simple) tiling whose automorphism group of isometries of $\mathbb{E}^n$ is crystallographic.
\end{Def}

\noindent Note that we restrict our investigations to \textit{simple tilings}, that is the tiles are assumed to be convex polytopes meeting facet-to-facet, and there is only a finite number of prototiles, up to isometries (not only translations). For a further discussion of this assumption see the beginning of Section~\ref{tilings-sec}.

\noindent Second, it turns out that even the strongest equivalence relation developed for tilings respectively the metric hulls of all their translates, that is \textit{mutual local derivability}, is not able to distinguish between two tilings with two different (non-conjugated) crystallographic groups as automorphism groups: We discuss a basic instance of this phenomenon in Ex.~\ref{MLD-trans-ex}. The remedy for this unsatisfactory situation is to introduce more isometries than just translations into the definition of local derivability (see Sec.~\ref{tiling-defs-ssec} for notation):
\begin{Def}[$=$ Definition~\ref{LD-def}]
Let $\gamma$ be an isometry of $\mathbb{E}^n$. A simple tiling $T^\prime$ of $\mathbb{E}^n$ is \textit{$\gamma$-locally derivable from} a simple tiling $T$ of $\mathbb{E}^n$ if there exists a real number $R > 0$ such that for all $x \in \mathbb{E}^n$ and $\phi \in \mathrm{Isom}(\mathbb{E}^n)$,
\[ [T]_{B_R(x)} = [\phi(T)]_{B_R(x)}\ \Rightarrow [T^\prime]_{\{\gamma(x)\}} = [\gamma\phi\gamma^{-1}T^\prime]_{\{\gamma(x)\}}. \]
\end{Def}

\noindent Conjugation with an isometry $\gamma$ is necessary to overcome the non-commutativity of $\mathrm{Isom}(\mathbb{E}^n)$ when proving analogues to statements on local derivability defined by using only translations (see Sec.~\ref{LD-ssec}). But it also makes sense geometrically: A rotated or reflected tiling still "looks the same" if one changes the point of view accordingly.

\noindent Next, moving around tilings by general isometries and not only shifting them by translations when defining local derivability requires similar changes in the definition of hulls of tilings and tiling spaces in general: They must be invariant under the action of $\mathrm{Isom}(\mathbb{E}^n)$, not just of $\mathrm{Trans}(\mathbb{E}^n)$. Furthermore, the tiling spaces must be closed with respect to a metric also taking into account general isometries:
\begin{Def}[$=$ Definition~\ref{dist-tilings-def}]
Let $T, T^\prime$ be two simple tilings of $\mathbb{E}^n$. Define $R(T, T^\prime) > 0$ to be the supremum of all $r>0$ such that there exist $\phi, \psi \in \mathrm{Isom}(\mathbb{E}^n)$ satisfying $d_O(\phi, \mathbbm{1}_{\mathbb{E}^n}), d_O(\psi, \mathbbm{1}_{\mathbb{E}^n}) < \frac{1}{2r}$ and
\[ \left[ \phi(T) \right]_{B_r(O)} = \left[ \psi(T^\prime) \right]_{B_r(O)}. \]
Then the \textit{distance} between $T$ and $T^\prime$ is set to be
\[ d_O(T, T^\prime) := \min \{\ln(3/2), \ln(1+\frac{1}{R(T,T^\prime)})\}. \]
\end{Def}

\noindent Here, $d_O$ also denotes the product metric on $\mathrm{Isom}(\mathbb{E}^n) = \mathrm{Trans}(\mathbb{E}^n) \rtimes O(\mathbb{E}^n_O)$ where we use the standard Euclidean metric on $\mathrm{Trans}(\mathbb{E}^n)$ and the operator norm on $O(\mathbb{E}^n_O)$, with fixed origin $O$. The properties of $d_O$, in particular its dependence on the choice of $O$, are discussed in Sec.~\ref{cryst-groups-ssec}.

\noindent In Sections~\ref{tilings-sec} and \ref{equiv-tilings-sec} we work out all the details of these constructions and prove their natural properties. In Section~\ref{cryst-tilings-sec} we use them to prove our main results on crystallographic tilings:
\begin{thm}[$=$ Theorem~\ref{MLD-cryst-thm}]
Two crystallographic tilings of $\mathbb{E}^n$ are $\gamma$-MLD if and only if their automorphism groups are conjugated by the isometry $\gamma$ of $\mathbb{E}^n$.
\end{thm}

\begin{thm}[$=$ Theorem~\ref{cryst-group-tiling-thm}]
For a given crystallographic group $\Gamma \subset \mathrm{Isom}(\mathbb{E}^n)$ there exists a simple tiling $T$ of $\mathbb{E}^n$ such that $\mathrm{Aut}(T) = \Gamma$.
\end{thm}

\noindent For wallpaper groups and space groups, $2$- and $3$-dimensional tilings with one of these groups as automorphism group are often constructed by \textit{decorating} the tiles, as is the case with actual wallpapers. These decorations kill additional automorphisms that the undecorated tiling might have. In the proof of Thm.~\ref{cryst-group-tiling-thm} we present a construction method that kills further automorphisms by subdividing tiles in a sufficiently general way. The tiles thus obtained may not look too regular but they are convex polytopes meeting facet-to-facet, so the tilings are simple. Allowing more general tiles, maybe not meeting facet-to-facet, leads to "nicer" tilings, see the many illustrations in \cite{GS89}.

\noindent Thm.~\ref{MLD-cryst-thm} and \ref{cryst-group-tiling-thm} together with the fact that the hull of a crystallographic tiling $T$ of $\mathbb{E}^n$ is homeomorphic to $\mathrm{Isom}(\mathbb{E}^n)/\mathrm{Aut}(T)$, preserving the natural group action of $\mathrm{Isom}(\mathbb{E}^n)$, yield a complete picture of the theory of crystallographic tilings. But crystallographic tilings can be used as a starting point to construct more general tilings. For example, the authors are preparing a continuation of this paper establishing the cut-and-project method for crystallographic tilings, thus generalizing the classical approach starting from lattice tilings (see the monograph \cite{FHK02}). Keeping in line with the framework in which Thm.~\ref{MLD-cryst-thm} and \ref{cryst-group-tiling-thm} hold the more general notion of MLD equivalence as given in Def.~\ref{LD-def} should be used. So, as when starting from lattice tilings the cut-and-project method using more general crystallographic tilings will yield simple tilings of $\mathbb{E}^n$ whose hulls are foliated, but by $\mathrm{Isom}(\mathbb{E}^n)$-orbits, not $\mathrm{Trans}(\mathbb{E}^n)$-orbits.

\noindent Thus, when studying the dynamics of these foliation we must face the additional difficulty of $\mathrm{Isom}(\mathbb{E}^n)$ being non-commutative, in contrast to $\mathrm{Trans}(\mathbb{E}^n) \cong \mathbb{R}^n$ used in the classical lattice setting. However, the group of isometries of $\mathbb{E}^n$ is well-behaved with respect to the techniques coming from ergodic theory and dynamical systems, for example the use of crossed product $C^\ast$-algebras to describe quotient spaces of group actions on a topological space.

\noindent So the hope is that using $\mathrm{Isom}(\mathbb{E}^n)$ instead of $\mathrm{Trans}(\mathbb{E}^n)$ opens up a new field of investigations on (the dynamics of) cut-and-project tilings, and possibly beyond.

\section{Preliminaries}

\subsection{Crystallographic Groups} \label{cryst-groups-ssec} There are numerous accounts of the theory of crystallographic groups; we follow the presentation of Farkas \cite{Far81}. In this subsection we collect definitions and statements used later on to fix the widely varying notation:

\noindent Let $\mathbb{E}^n$ denote the set of all points $(x_1, \ldots, x_n) \in \mathbb{R}^n$. \textit{Translations} of $\mathbb{E}^n$ are considered as permutations of the points in $\mathbb{E}^n$. In particular, for every two points $P = (p_1, \ldots, p_n), Q = (q_1, \ldots, q_n) \in \mathbb{E}^n$ there is a unique translation $\overrightarrow{PQ}$ mapping $(x_1, \ldots, x_n)$ to
\[ (x_1 + q_1 - p_1, \ldots, x_n+q_n - p_n). \]
Vice versa, for every point $P \in \mathbb{E}^n$ a translation $\tau$ is equal to $\overrightarrow{P\tau(P)}$. The set $\mathrm{Trans}(\mathbb{E}^n)$ of all translations has the natural structure of an $n$-dimensional vector space, where the addition is just composition of translations. For every point $P \in \mathbb{E}^n$ the natural bijective map $\mathrm{Trans}(\mathbb{E}^n) \rightarrow \mathbb{E}^n$ given by $\tau \mapsto \tau(P)$ induces a vector space structure on $\mathbb{E}^n$, and the resulting vector space is denoted by $\mathbb{E}^n_P$.

\noindent Using the same bijective map, an inner product and the corresponding norm $\parallel . \parallel$ on the vector space $\mathrm{Trans}(\mathbb{E}^n)$ induce an inner product and a norm on $\mathbb{E}^n_P$ and thus a distance on $\mathbb{E}^n$, by setting
\[ d(Q,Q^\prime) = \norm{\overrightarrow{QQ^\prime}},\ \mathrm{for\ points\ } Q, Q^\prime \in \mathbb{E}^n. \]
Later on, we will denote the ball with center in $x \in \mathbb{E}^n$ and radius $R > 0$ with respect to this distance by $B_R(x)$.

\noindent Given a point $O \in \mathbb{E}^n$ every permutation of $\mathbb{E}^n$ preserving the distance can be written as a product of a translation with a linear transformation in the orthogonal group $O(\mathbb{E}^n_O)$. The group of all such \textit{isometries} is denoted by $\mathrm{Isom}(\mathbb{E}^n)$.

\begin{thm}
$\mathrm{Trans}(\mathbb{E}^n)$ is a normal subgroup of $\mathrm{Isom}(\mathbb{E}^n)$, and $\mathrm{Trans}(\mathbb{E}^n) \cap O(\mathbb{E}^n_O) = \{\mathbf{1}_{\mathbb{E}^n}\}$. In particular, $\mathrm{Isom}(\mathbb{E}^n)$ is isomorphic to the semidirect product $\mathrm{Trans}(\mathbb{E}^n) \rtimes O(\mathbb{E}^n_O)$.
\end{thm}
\begin{proof}
This is well known. We just note the formula
\begin{equation} \label{conj-trans-form}
\alpha \cdot \tau \cdot \alpha^{-1} = \overrightarrow{O\alpha(\tau(O))}
\end{equation}
for a translation $\tau$ and an orthogonal map $\alpha \in O(\mathbb{E}^n_O)$: Via the isomorphism $\mathrm{Trans}(\mathbb{E}^n) \rightarrow \mathbb{E}^n_O$ the point $\alpha \cdot \tau \cdot \alpha^{-1}(P)$ corresponds to
\[ \alpha \left( \alpha^{-1}(\overrightarrow{OP}) + \overrightarrow{O\tau(O)} \right) = \overrightarrow{OP} + \overrightarrow{O\alpha(\tau(O))} \]
corresponding to $\overrightarrow{O\alpha(\tau(O))}(P)$.
\end{proof}

\noindent Since $\mathrm{Isom}(\mathbb{E}^n)$ is a semidirect product of $\mathrm{Trans}(\mathbb{E}^n)$ and  $O(\mathbb{E}^n_O)$ we can define a distance on $\mathrm{Isom}(\mathbb{E}^n)$ by using the (Euclidean) metric $\norm{\tau}_{\mathrm{Eucl}}$ on translations $\tau$ introduced above, and the operator norm on $\mathrm{End}(\mathbb{E}^n_0)$, the vector space of not necessarily invertible linear transformations of $\mathbb{E}^n_0$, given by
\[ \norm{\alpha}_{\mathrm{op}} := \max \left\{\norm{\overrightarrow{O\alpha(P)}}_{\mathrm{Eucl}}: \norm{\overrightarrow{OP}}_{\mathrm{Eucl}} \leq 1\right\} \]
for $\alpha \in \mathrm{End}(\mathbb{E}^n_0)$: For products $\tau_O \cdot \alpha_O, \sigma_O \cdot \beta_O \in \mathrm{Isom}(\mathbb{E}^n)$, with $\tau_O, \sigma_O \in \mathrm{Trans}(\mathbb{E}^n)$ and $\alpha_O, \beta_O \in O(\mathbb{E}^n_O)$, we can set
\[ d_O(\tau_O \cdot \alpha_O, \sigma_O \cdot \beta_O) := \norm{\tau_O - \sigma_O}_{\mathrm{Eucl}} + \norm{\alpha_O - \beta_O}_{\mathrm{op}} \]
because the product decomposition is unique.

\noindent For later purposes we state some properties of the operator norm $\norm{\cdot}_{\mathrm{op}}$ on $\mathrm{End}(\mathbb{E}^n_0)$ and the distance $d_O$ on $\mathbb{E}^n$:
\begin{lem} \label{d_O-lem}
In the following, $\alpha, \beta$ are orthogonal maps on $\mathbb{E}^n_O$, $\sigma, \tau$ are translations of $\mathbb{E}^n$, $\chi, \phi, \psi$ are isometries of $\mathbb{E}^n$, and $\mathbbm{1}_{\mathbb{E}^n}$ is the identity map on $\mathbb{E}^n$. Then:
\begin{enumerate}
\item $\norm{\alpha}_{\mathrm{op}} = 1$.
\item $\norm{\alpha \cdot \beta}_{\mathrm{op}} \leq \norm{\alpha}_{\mathrm{op}} \cdot \norm{\beta}_{\mathrm{op}}$ if $\alpha, \beta \in \mathrm{End}(\mathbb{E}^n_0)$, and equality holds if $\alpha \in O(\mathbb{E}^n_O)$.
\item $\norm{\alpha\beta\alpha^{-1} - \mathbbm{1}_{\mathbb{E}^n}}_{\mathrm{op}} =
           \norm{\beta - \mathbbm{1}_{\mathbb{E}^n}}_{\mathrm{op}}$.
\item $\norm{\alpha\beta - \mathbbm{1}_{\mathbb{E}^n}}_{\mathrm{op}} \leq
           \norm{\alpha - \mathbbm{1}_{\mathbb{E}^n}}_{\mathrm{op}} +
           \norm{\beta - \mathbbm{1}_{\mathbb{E}^n}}_{\mathrm{op}}$.
\item $d_O( \alpha \cdot \sigma, \mathbbm{1}_{\mathbb{E}^n} ) = d_O( \sigma \cdot \alpha, \mathbbm{1}_{\mathbb{E}^n} )$.
\item \label{inv-met-for}$d_O(\chi ^{-1}, \mathbbm{1}_{\mathbb{E}^n} )  = d_O( \chi  , \mathbbm{1}_{\mathbb{E}^n} )$.
\item \label{prod-id-for}$d_O( \chi \cdot \phi, \mathbbm{1}_{\mathbb{E}^n}) \leq d_O( \chi , \mathbbm{1}_{\mathbb{E}^n} )+
                                                                                                     d_O(\phi , \mathbbm{1}_{\mathbb{E}^n} )$.
\item $d_O(\sigma,\tau)= \norm{\sigma -\tau}_{Eucl}$.
\item \label{conj-met-for} $d_{\sigma_O(O)}(\chi \phi \chi^{-1},\chi \psi\chi^{-1})=d_ O(\phi,\psi)$ if $\chi = \sigma_O \cdot \alpha_O$ with $\sigma_O \in \mathrm{Trans}(\mathbb{E}^n)$ and $\alpha_O \in O(\mathbb{E}^n_O)$.
\item \label{met-comp-for} $d_{\sigma(O)}(\phi,\psi) \leq (1 + \norm{\sigma}_{\mathrm{Eucl}}) d_O(\phi, \psi)$.
\item \label{met-conj-bd-for} $d_O(\chi \phi \chi ^{-1},\mathbbm{1}_{\mathbb{E}^n} ) \leq
           d_O(\phi, \mathbbm{1}_{\mathbb{E}^n})\left(1+d_O(\chi, \mathbbm{1}_{\mathbb{E}^n})\right)$.
\item \label{met-prod0-for} $d_O(\phi \chi, \phi) = d_O(\chi, \mathbbm{1}_{\mathbb{E}^n})$.
\item \label{met-prod-for} $d_O(\chi \phi, \phi) \leq d_O(\chi, \mathbbm{1}_{\mathbb{E}^n}) \cdot
                                            (1 + d_O(\phi, \mathbbm{1}_{\mathbb{E}^n}))$.
\item \label{met-prod2-for} $d_O(\chi, \mathbbm{1}_{\mathbb{E}^n}) \leq d_O(\chi \phi, \phi) \cdot
                                            (1 + d_O(\phi, \mathbbm{1}_{\mathbb{E}^n}))$.
\end{enumerate}
\end{lem}
\begin{proof}
(1) and (2) are obvious from the definitions. For (3), use
\[ \alpha\beta\alpha^{-1} - \mathbbm{1}_{\mathbb{E}^n} =
    \alpha \cdot (\beta - \mathbbm{1}_{\mathbb{E}^n}) \cdot \alpha^{-1} \]
and (2). For (4), expand
\[ \alpha\beta - \mathbbm{1}_{\mathbb{E}^n} = \alpha\beta - \alpha + \alpha - \mathbbm{1}_{\mathbb{E}^n} =
    \alpha(\beta - \mathbbm{1}_{\mathbb{E}^n}) + \alpha - \mathbbm{1}_{\mathbb{E}^n} \]
and use the triangle inequality of the operator norm, (2) and (1).

\noindent (5) follows from $\alpha \cdot \sigma = \overrightarrow{O\alpha(\sigma(O))} \cdot \alpha$, using formula~(\ref{conj-trans-form}), and $\norm{\overrightarrow{O\alpha(\sigma(O))}}_{\mathrm{Eucl}} = \norm{\sigma}_{\mathrm{Eucl}}$ since $\alpha$ is orthogonal with center in $O$.

\noindent For (6) we write $\chi = \sigma \cdot \alpha$ as a product of a translation and an orthogonal map. Then $\chi^{-1} = \alpha^{-1} \sigma^{-1}$, and we use (5) together with $\norm{\sigma^{-1}}_{\mathrm{Eucl}} = \norm{-\sigma}_{\mathrm{Eucl}} = \norm{\sigma}_{\mathrm{Eucl}}$ and
\[ \norm{\alpha^{-1} - \mathbbm{1}_{\mathbb{E}^n}}_{\mathrm{op}} = \norm{\alpha^{-1}}_{\mathrm{op}} \cdot
   \norm{\mathbbm{1}_{\mathbb{E}^n} - \alpha}_{\mathrm{op}} =
   \norm{\alpha - \mathbbm{1}_{\mathbb{E}^n}}_{\mathrm{op}}. \]
If $\chi = \sigma \cdot \alpha$ and $\psi = \tau \cdot \beta$ as products of a translation and an orthogonal map then by~(\ref{conj-trans-form}), $\chi\phi = \sigma\overrightarrow{O\alpha(\tau(O))} \cdot \alpha\beta$.
Hence, (7) is implied by the triangle inequality of the norm $\norm{\cdot}_{\mathrm{Eucl}}$, (4)  and
\[ \norm{\overrightarrow{O\alpha(\tau(O))}}_{\mathrm{Eucl}} = \norm{\tau}_{\mathrm{Eucl}}. \]
(8) is obvious. For (9) and (\ref{met-comp-for}) let $\phi = \tau_1\beta_1$ and $\psi = \tau_2\beta_2$ be the product decomposition into translations and orthogonal maps centered in $O$.

\noindent To show (9) we assume first that $\chi = \sigma$ is a translation. Then
\begin{eqnarray*}
d_{\sigma(O)}(\sigma \phi \sigma^{-1},\sigma \psi\sigma^{-1}) & = &
d_{\sigma(O)}(\tau_1 \sigma \beta_1 \sigma^{-1}, \tau_2 \sigma \beta_2 \sigma^{-1}) = \\
 & = & \norm{\tau_1 - \tau_2}_{\mathrm{Eucl}} + \norm{\sigma \beta_1 \sigma^{-1} - \sigma \beta_2 \sigma^{-1}}_{\mathrm{op}} = \\
 & = & \norm{\tau_1 - \tau_2}_{\mathrm{Eucl}} + \norm{\beta_1 - \beta_2}_{\mathrm{op}} = d_O(\phi, \psi),
\end{eqnarray*}
since $\sigma \beta_1 \sigma^{-1}, \sigma \beta_2 \sigma^{-1}$ are orthogonal maps in $O(\mathbb{E}^n_{\sigma(O)})$ conjugated to $\beta_1, \beta_2$.

\noindent So we can reduce (9) to the case when $\chi \in O(\mathbb{E}^n_O)$. Then
\begin{eqnarray*}
d_O(\chi\phi\chi^{-1}, \chi\psi\chi^{-1}) & = & d_O(\chi\tau_1\beta_1\chi^{-1}, \chi\tau_2\beta_2\chi^{-1}) = \\
 & = & d_O(\overrightarrow{O\chi(\tau_1(O))}\chi\beta_1\chi^{-1}, \overrightarrow{O\chi(\tau_2(O))}\chi\beta_2\chi^{-1}) = \\
 & = & \norm{\overrightarrow{O\chi(\tau_1(O))} - \overrightarrow{O\chi(\tau_2(O))}}_{\mathrm{Eucl}} +
           \norm{\chi\beta_1\chi^{-1} - \chi\beta_2\chi^{-1}}_{\mathrm{op}} = \\
 & = &  \norm{\tau_1 - \tau_2}_{\mathrm{Eucl}} + \norm{\beta_1 - \beta_2}_{\mathrm{op}} = d_O(\phi, \psi).
\end{eqnarray*}
(9) implies $d_{\sigma(O)}(\phi, \psi) = d_O(\sigma^{-1}\phi\sigma, \sigma^{-1}\psi\sigma)$, hence (\ref{met-comp-for}) follows from
\begin{eqnarray*}
d_O(\sigma^{-1}\tau_1\beta_1\sigma, \sigma^{-1}\tau_2\beta_2\sigma) & = &
d_O(\sigma^{-1}\tau_1\overrightarrow{O\beta_1(\sigma(O))}\beta_1, \sigma^{-1}\tau_2\overrightarrow{O\beta_2(\sigma(O))}\beta_2) \leq \\
 & \leq & \norm{\tau_1 - \tau_2}_{\mathrm{Eucl}} + \norm{\overrightarrow{O\beta_1(\sigma(O))} - \overrightarrow{O\beta_2(\sigma(O))}}_{\mathrm{Eucl}} + \norm{\beta_1 - \beta_2}_{\mathrm{op}} \leq \\
 & \leq & \norm{\tau_1 - \tau_2}_{\mathrm{Eucl}} + \norm{\beta_1 - \beta_2}_{\mathrm{op}} \cdot \norm{\sigma}_{\mathrm{Eucl}}
           + \norm{\beta_1 - \beta_2}_{\mathrm{op}} \leq \\
 & \leq & (1 + \norm{\sigma}_{\mathrm{Eucl}})d_O(\phi, \psi).
\end{eqnarray*}
If $\chi = \sigma \cdot \alpha$ with $\sigma \in \mathrm{Trans}(\mathbb{E}^n)$ and $\alpha \in O(\mathbb{E}^n_O)$ then (9) also implies that $d_O(\chi\phi\chi^{-1}, \mathbbm{1}_{\mathbb{E}^n}) = d_{\sigma^{-1}(O)}(\phi, \mathbbm{1}_{\mathbb{E}^n})$, and by (\ref{met-comp-for}) this is $\leq (1 + \norm{\sigma^{-1}}_{\mathrm{Eucl}})d_O(\phi, \mathbbm{1}_{\mathbb{E}^n})$. Then (11) follows from $\norm{\sigma^{-1}}_{\mathrm{Eucl}} = \norm{\sigma}_{\mathrm{Eucl}} \leq d_O(\chi, \mathbbm{1}_{\mathbb{E}^n})$.

\noindent For (\ref{met-prod0-for}) write $\phi = \sigma \alpha$, $\chi = \tau \beta$,
$\sigma, \tau \in \mathrm{Trans}(\mathbb{E}^n)$, $\alpha, \beta \in O(\mathbb{E}^n_O)$. By (\ref{conj-trans-form}) we know that $\phi \chi = \sigma \alpha \tau \beta = \sigma \overrightarrow{O\alpha(\tau(O))} \alpha \beta$. Consequently,
\begin{eqnarray*}
d_O(\phi \chi, \phi) & = & \norm{\sigma + \overrightarrow{O\alpha(\tau(O))} - \sigma}_{\mathrm{Eucl}} + \norm{\alpha \beta - \alpha}_{\mathrm{op}} = \norm{\tau}_{\mathrm{Eucl}} + \norm{\beta - \mathbbm{1}_{\mathbb{E}^n}}_{\mathrm{op}} \\
 & = & d_O(\chi, \mathbbm{1}_{\mathbb{E}^n}).
\end{eqnarray*}
From (\ref{met-prod0-for}) we can deduce (\ref{met-prod-for}) and (\ref{met-prod2-for}):
\begin{eqnarray*}
d_O(\chi \phi, \phi) & = & d_O(\phi^{-1}\phi\chi\phi, \phi^{-1}\phi\phi) \stackrel{(\ref{conj-met-for})}{=}
                                         d_{\phi(O)}(\phi \chi, \phi) \\
 & \stackrel{(\ref{met-comp-for})}{\leq} & (1 + d_O(\phi, \mathbbm{1}_{\mathbb{E}^n})) \cdot d_O(\phi \chi, \phi)
    \stackrel{(\ref{met-prod0-for})}{=} (1 + d_O(\phi, \mathbbm{1}_{\mathbb{E}^n})) \cdot
                                                              d_O(\chi, \mathbbm{1}_{\mathbb{E}^n})
\end{eqnarray*}
and
\begin{eqnarray*}
d_O(\chi, \mathbbm{1}_{\mathbb{E}^n}) & \stackrel{(\ref{met-prod0-for})}{=} & d_O(\phi \chi, \phi)
                \stackrel{(\ref{conj-met-for})}{=} d_{\phi^{-1}(O)}(\phi^{-1}\phi\chi\phi, \phi^{-1}\phi\phi) \\
 & \stackrel{(\ref{met-comp-for}), (\ref{inv-met-for})}{\leq} & d_O(\chi \phi, \phi) \cdot
                                            (1 + d_O(\phi, \mathbbm{1}_{\mathbb{E}^n})).
\end{eqnarray*}
\end{proof}

\begin{cor} \label{cont-conj-mult-cor}
The topology on $\mathrm{Isom}(\mathbb{E}^n)$ induced by the metric $d_O$ makes conjugation and multiplication from left or right by an element $\gamma \in \mathrm{Isom}(\mathbb{E}^n)$ continuous.
\end{cor}
\begin{proof}
$\phi \mapsto \gamma \phi \gamma^{-1}$ defines a continuous map on $\mathrm{Isom}(\mathbb{E}^n)$ by Lem.~\ref{d_O-lem}(\ref{conj-met-for}) and~(\ref{met-comp-for}).

\noindent $\phi \mapsto \gamma \phi$ defines a continuous map on $\mathrm{Isom}(\mathbb{E}^n)$ because for all $\phi_n, \phi \in \mathrm{Isom}(\mathbb{E}^n)$,  Lem.~\ref{d_O-lem}(\ref{met-prod0-for}) shows that
\[ d_O(\phi_n, \phi) = d_O(\phi\phi^{-1}\phi_n, \phi) = d_O(\phi^{-1}\phi_n, \mathbbm{1}_{\mathbb{E}^n}) =
   d_O(\gamma \phi \phi^{-1}\phi_n, \gamma \phi) = d_O(\gamma \phi_n, \gamma \phi), \]
hence $d_O(\phi_n, \phi) \rightarrow 0$ implies $d_O(\gamma \phi_n, \gamma \phi) \rightarrow 0$ for $n \rightarrow \infty$.

\noindent $\phi \mapsto \phi \gamma$ defines a continuous map on $\mathrm{Isom}(\mathbb{E}^n)$ because by Lem.~\ref{d_O-lem}(\ref{conj-met-for}) and the same argument as before,
\[ d_O(\phi_n \gamma, \phi \gamma) = d_{\gamma(O)}(\gamma \phi_n, \gamma \phi) = d_{\gamma(O)}(\phi_n, \phi) \]
for all $\phi_n, \phi \in \mathrm{Isom}(\mathbb{E}^n)$, hence Lem.~\ref{d_O-lem}(\ref{met-comp-for}) implies that $d_O(\phi_n \gamma, \phi \gamma) \rightarrow 0$ if $d_O(\phi_n, \phi) \rightarrow 0$ for $n \rightarrow \infty$.
\end{proof}

\noindent The metric $d_O$ depends on the choice of $O$, but (\ref{met-comp-for}) shows that the topologies induced by $d_O$ and $d_{O^\prime}$ for different points $O, O^\prime \in \mathbb{E}^n$ are equal. Topological notions connected to $\mathrm{Isom}(\mathbb{E}^n)$ can thus be defined independently of the choice of~$O$.

\noindent For lack of reference we also prove the following fact on square roots of orthogonal maps on $\mathbb{E}^n_O$:
\begin{lem} \label{sqrt-lem}
For every $\alpha \in O(\mathbb{E}^n_O)$ with $\norm{\alpha - \mathbbm{1}_{\mathbb{E}^n}}_{\mathrm{op}} \leq 1$ there exists $\beta \in  O(\mathbb{E}^n_O)$ such that $\beta^2 = \alpha$ and
\[ \norm{\beta - \mathbbm{1}_{\mathbb{E}^n}}_{\mathrm{op}} \leq
    \norm{\alpha - \mathbbm{1}_{\mathbb{E}^n}}_{\mathrm{op}}. \]
\end{lem}
\begin{proof}
By standard Linear Algebra (see \cite[p.91pp]{Axl97}) every $\alpha \in O(\mathbb{E}^n_O)$ can be represented by a matrix of the form
\[ A = \left( \begin{array}{cccc} M_1 & 0 & \cdots & 0 \\
                                                    0 & \ddots & 0 & \vdots \\
                                                    \vdots & 0 & \ddots & 0 \\
                                                    0 & \cdots & 0 & M_r \end{array}\right) \]
in terms of an orthonormal basis of $\mathbb{E}^n_O$ where $M_i$ is either one of the $1 \times 1$-matrices $(1)$ or $(-1)$, or the rotation matrix $R(\theta) = \begin{pmatrix} \cos \theta & - \sin \theta \\ \sin \theta & \cos \theta \end{pmatrix}$ associated to a rotation angle $\theta \in (-\pi, \pi)$. Because of the orthonormality of the basis $\norm{\alpha - \mathbbm{1}_{\mathbb{E}^n}}_{\mathrm{op}}$ is the sum of the operator norms of $M_i - \mathbbm{1}$, hence $\norm{\alpha - \mathbbm{1}_{\mathbb{E}^n}}_{\mathrm{op}} \leq 1$ excludes $M_i = (-1)$. Then we can choose $\beta$ to be represented in the same orthonormal basis by
\[ B = \left( \begin{array}{cccc} N_1 & 0 & \cdots & 0 \\
                                                    0 & \ddots & 0 & \vdots \\
                                                    \vdots & 0 & \ddots & 0 \\
                                                    0 & \cdots & 0 & N_r \end{array}\right) \]
where $N_i = (1)$ if $M_i = (1)$ and $N_i = R(\theta/2)$ if $M_i = R(\theta)$. A standard calculation shows that
\begin{eqnarray*}
\norm{R(\theta)-\mathbbm{1}_{\mathbb{R}^2}}^2_{\mathrm{op}} & = &
\max_{(x, y) \in S^1} \norm{(R(\theta)-\mathbbm{1}_{\mathbb{R}^2})\begin{pmatrix} x \\ y \end{pmatrix}}^2_{\mathrm{Eucl}}
  = 2 - 2\cos \theta \geq \\
 & \geq & 2 - 2 \cos \theta/2 = \norm{R(\theta/2)-\mathbbm{1}_{\mathbb{R}^2}}^2_{\mathrm{op}}
\end{eqnarray*}
if $\theta \in (-\pi, \pi)$, hence $\norm{\beta - \mathbbm{1}_{\mathbb{E}^n}}_{\mathrm{op}} \leq
    \norm{\alpha - \mathbbm{1}_{\mathbb{E}^n}}_{\mathrm{op}}$.
\end{proof}

\begin{Def}
A \textit{crystallographic group} $\Gamma$ is a subgroup of the group $\mathrm{Isom}(\mathbb{E}^n)$ such that $\Gamma$ is discrete and $\mathrm{Isom}(\mathbb{E}^n)/\Gamma$ is compact.
\end{Def}

\begin{thm}[Bieberbach 1912, {\cite[Thm.14]{Far81}}] \label{Bieb-thm}
Let $\Gamma \subset \mathrm{Isom}(\mathbb{E}^n)$ be a crystallographic group. Then:
\begin{enumerate}
\item $\Gamma \cap \mathrm{Trans}(\mathbb{E}^n)$ is a lattice of full rank in $\mathrm{Trans}(\mathbb{E}^n)$.
\item The point group $\Gamma / (\Gamma \cap \mathrm{Trans}(\mathbb{E}^n))$ of $\Gamma$ is finite.
\end{enumerate}
\end{thm}

\begin{Def}
A crystallographic group $\Gamma \subset \mathrm{Isom}(\mathbb{E}^n)$ is called \textit{symmorphic} (with respect to a point $P \in \mathbb{E}^n$) if there is a finite subgroup $G \subset \mathrm{Isom}(\mathbb{E}^n)$ isomorphic to the point group $\Gamma / (\Gamma \cap \mathrm{Trans}(\mathbb{E}^n))$ such that
\[ \Gamma \cong (\Gamma \cap \mathrm{Trans}(\mathbb{E}^N)) \rtimes G \subset \mathrm{Trans}(\mathbb{E}^n) \rtimes O(\mathbb{E}^N_P) \cong \mathrm{Isom}(\mathbb{E}^n). \]
\end{Def}

\begin{prop} \label{symmorph-prop}
For every crystallographic group $\Gamma \subset \mathrm{Isom}(\mathbb{E}^n)$ there exists a point $P \in \mathbb{E}^n$ and a crystallographic group $\Gamma^\ast \subset \mathrm{Isom}(\mathbb{E}^n)$ symmorphic with respect to $P$ such that
$\Gamma \subset \Gamma^\ast$ and the point groups of $\Gamma$ and $\Gamma^\ast$ are isomorphic.
\end{prop}
\begin{proof}
See \cite[p.534p.]{Far81}.
\end{proof}

\subsection{Convex Polytopes}

\noindent We fix some notation on convex polytopes, following \cite{Gru03}:
\begin{Def}
A \textit{convex polytope} in $\mathbb{R}^n$ is the \textit{convex hull} of a finite set of points $\{p_1, \ldots, p_k\}$, that is, the set of points
\[ \langle p_1, \ldots, p_k \rangle := \left\{ p = \sum_{i=1}^k t_i p_i : \sum_{i=1}^k t_i = 1 \right\} \subset \mathbb{R}^n. \]
We say that $p_1, \ldots, p_k$ \textit{minimally generate} the convex hull $\langle p_1, \ldots, p_k \rangle$ if the convex hull of a proper subset of $\{ p_1, \ldots, p_k \}$ is also a proper subset of $\langle p_1, \ldots, p_k \rangle$.
\end{Def}

\noindent It is easy to see that a convex polytope $P = \langle p_1, \ldots, p_k \rangle$ is a \textit{convex subset} of $\mathbb{R}^n$, that is, for all points $p, q \in P$ the points $t \cdot p + (1-t) \cdot q$, $0 \leq t \leq 1$, on the connecting line segment are also on $P$.

\noindent Every convex polytope $P \subset \mathbb{R}^n$ spans an affine subspace $L$, and the dimension of $P$ is the dimension of $L$. In particular, $\dim P = n$ if and only if $P$ has a non-empty open interior $P^o$.

\begin{Def}
Let $P = \langle p_1, \ldots, p_k \rangle$ be an $n$-dimensional convex polytope in $\mathbb{R}^n$, minimally generated by $p_1, \ldots, p_k$. Then for a subset of points $p_{i_1}, \ldots, p_{i_l}$, the convex polytope $F_{i_1, \ldots, i_l} := \langle p_{i_1}, \ldots, p_{i_l} \rangle$ is called an \textit{$m$-face} of $P$ if $F_{i_1, \ldots, i_l}$ is contained in the boundary $\partial P \subset P$ and $\dim F_{i_1, \ldots, i_l} = m$.

\noindent In particular, the points $p_1, \ldots, p_k$ are called the \textit{vertices} of $P$, $2$-faces are called \textit{edges} and $(n-1)$-faces are called \textit{facets} of $P$.
\end{Def}

\noindent Dually, a convex polytope $P \subset \mathbb{E}^n$ can be described as the intersection of finitely many \textit{half-spaces}
\[ H := \{(x_1, \ldots, x_n) \in \mathbb{E}^n: a_0 \leq a_1x_1 + \cdots + a_nx_n \} \]
with $a_0, a_1, \ldots, a_n \in \mathbb{R}$ and $(a_1, \ldots, a_n) \neq (0, \ldots, 0)$. If $H$ is a half-space such that $P \subset H$ and the affine hyperplane $E := \{(x_1, \ldots, x_n) \in \mathbb{E}^n: a_0 = a_1x_1 + \cdots + a_nx_n \}$ bounding $H$ intersects $P$ then $E$ is called a \textit{supporting hyperplane} of $P$. In that case, the intersection $P \cap E$ is a face of $P$.

\subsection{Point Sets}
Subsets of $\mathbb{E}^n$ with special properties are useful to construct tilings.
\begin{Def}
A \textit{point set} $X \subset \mathbb{E}^n$ is called \textit{relatively dense} if there exists a number $R > 0$ such that $B_R(y) \cap X \neq \emptyset$ for all $y \in \mathbb{E}^n$.

\noindent A point set $X \subset \mathbb{E}^n$ is called \textit{uniformly discrete} if there exists a number $r > 0$ such that $B_r(x) \cap X = \{x\}$ for all $x \in X$.

\noindent A point set $X \subset \mathbb{E}^n$ is called a \textit{Delone set} if $X$ is relatively dense and uniformly discrete.
\end{Def}

\section{Simple Tilings and Tiling Spaces} \label{tilings-sec}

\subsection{Definitions} \label{tiling-defs-ssec}

\noindent We restrict our attention to tilings made of convex polytopes as tiles. Of course, there are lots of quite regular tilings using non-convex polytopes or even more general tiles (see the book \cite{GS89}). But it is widely believed (however rarely proven in particular situations) that these more general tilings are equivalent (in whatever sense prefered) to tilings with convex polytopes, for example by extracting a point set from the general tiling and then applying the Voronoi cell tiling construction on the point set (see below). Voronoi cell tilings are automatically tilings by convex polytopes, and furthermore these convex polytopes meet facet-to-facet. Thus we also include the latter property into our definition of a tiling.

\noindent For all the definitions below we identify the points of $\mathbb{R}^n$ with those of the affine space $\mathbb{E}^n$.

\begin{Def}
A set of convex polytopes $\{t_i\}_{i \in I}$ is called a \textit{tiling} $T$ of $\mathbb{E}^n$ if
\begin{enumerate}
\item $\bigcup_{i \in I} t_i = \mathbb{E}^n$ and
\item for all $i, j \in I$ the intersection $t_i \cap t_j$ is a face of both $t_i$ and $t_j$. In particular, if $t_i \cap t_j$ is $(n-1)$-dimensional, the two tiles meet full-facet to full-facet.
\end{enumerate}
A \textit{patch} of a tiling $T$ is a subset of the tiles in $T$.
\end{Def}

\noindent If $A \subset \mathbb{E}^n$ is a bounded subset then $[T]_A$ denotes the patch of all tiles $t$ in a tiling $T$ of $\mathbb{E}^n$  intersecting $A$.

\noindent If $\phi$ is an isometry of $\mathbb{E}^n$ then for each tiling $T = \{t_i\}_{i \in I}$ of $\mathbb{E}^n$ the set $\phi(T) := \{\phi(t_i)\}_{i \in I}$ is also a tiling of $\mathbb{E}^n$. If the isometry is a translation $\tau$ in $\mathrm{Trans}(\mathbb{E}^n)$ we also write $T + \tau$ for the shifted tiling.

\begin{const}
Let $X \subset \mathbb{E}^n$ be a point set. To each point $x_0 \in X$ we associate the \textit{Voronoi-cell}
\[ V_{x_0}(X) := \{y \in \mathbb{E}^n: \norm{y - x_0} \leq \norm{y - x}\ \mathrm{for\ all\ } x \in X\} \]
\textit{of} $x_0$ \textit{in} $X$.

\noindent For fixed $x \in X$ different from $x_0$, $\{y \in \mathbb{E}^n: \norm{y - x_0} \leq \norm{y - x}\}$ is the half-space $H_{x,x_0}$ bounded by the affine hyperplane $E_{x,x_0}$ perpendicular to the line through $x$ and $x_0$ and passing through the midpoint of the line segment from $x_0$ to $x$ that contains $x_0$. So
\[ V_{x_0}(X) = \bigcap_{x \in X} H_{x,x_0}, \]
and if finitely many of these half-spaces suffice to cut out the Voronoi-cell then $V_{x_0}(X)$ is a convex polytope.

\noindent However this need not hold for arbitrary point sets $X \subset \mathbb{E}^n$.
\end{const}

\begin{lem} \label{VC-simplex-lem}
Let $X = \{p_0, p_1, \ldots, p_{n+1}\} \subset \mathbb{E}^n$ be a point set such that the convex hull of $p_1,\ldots,  p_{n+1}$ is an $n$-simplex $\Delta \subset \mathbb{E}^n$ and $p_0 \in \Delta^o$. Then $V_{p_0}(X)$ is a convex polytope.
\end{lem}
\begin{proof}
By definition, $V_{p_0}(\{p_0, p_1, \ldots, p_{n+1}\}) = \bigcap_{i=1}^{n+1} H_{p_i, p_0}$.

\noindent The convex hull of $p_1, \ldots, p_{n+1}$ will be an $n$-simplex if and only if $p_1-p_{n+1}, \ldots, p_n - p_{n+1}$ are linearly independent. $p_0 \in \Delta^o$ implies that also $p_1-p_0, \ldots, p_n - p_0$ are linearly independent, and $p_{n+1}-p_0$ is a linear combination of the $p_i-p_0$, with strictly negative coefficients.

\noindent Consequently, there is an affine-linear transformation of $\mathbb{E}^n$ such that
\[ p_0 = (0,\ldots,0), p_1 = (1,0, \ldots, 0), \ldots, p_n = (0, \ldots, 0, 1), p_{n+1} = (-a_1, \ldots, -a_n) \]
with $a_1, \ldots, a_n > 0$. Then for $i = 1, \ldots, n$
\[H_{p_i, p_0} = \{x_i = -\frac{1}{2}\}\ \mathrm{and}\ H_{p_{n+1}, p_0} = \{a_1x_1+ \cdots + a_nx_n \geq - \sum_{i=1}^n \frac{a_i^2}{2}\}. \]
This implies that $V_{p_0}(\{p_0, p_1, \ldots, p_{n+1}\}) = \bigcap_{i=1}^{n+1} H_{p_i, p_0}$ is bounded, hence a convex polytope.
\end{proof}

\begin{prop} \label{VT-prop}
Let $X \subset \mathbb{E}^n$ be a Delone point set. Then $\{V_x(X): x \in X\}$, the set of all Voronoi-cells of points $x \in X$ in $X$,   is a tiling of $\mathbb{E}^n$, called the Voronoi-cell tiling $VT(X)$ associated to $X$.
\end{prop}
\begin{proof}
For $x_0 \in X$ pick points $y_1, \ldots, y_{n+1} \in \mathbb{E}^n$ such that the convex hull of $y_1, \ldots, y_{n+1}$ is an $n$-simplex  whose interior contains $x_0$. Since $X$ is relatively dense there exists $r > 0$ such that $B_r(y_i) \cap X \neq \emptyset$ for $i = 1, \ldots, n+1$. Possibly dilating $\mathbb{E}^n$ with origin in $x_0$ we can pick points $x_i \in B_r(y_i) \cap X$ such that the convex hull of $x_1, \ldots, x_{n+1}$ is still an $n$-simplex $\Delta \subset \mathbb{E}^n$ whose interior contains $x_0$: Linear independence is an open and homogeneous condition on the coordinates of the points $x_1, \ldots, x_{n+1}$ with respect to the origin in $x_0$.

\noindent Applying Lem.~\ref{VC-simplex-lem} we conclude that $V_{x_0}(\{x_0, x_1, \ldots, x_{n+1}\})$ is a convex polytope. Since $X$ is uniformly discrete the set $X_0 \subset X$ of points $x$ such that $V_{x_0}(\{x_0, x_1, \ldots, x_{n+1}\}) \not\subset H_{x,x_0}$ is finite. Consequently, every Voronoi-cell
\[ V_{x_0}(X) = V_{x_0}(\{x_0, x_1, \ldots, x_{n+1}\}) \cap \bigcap_{x \in X_0} H_{x, x_0}, \]
in $\{V_x(X): x \in X\}$ is a convex polytope.

\noindent Uniform discreteness of $X$ also implies that the non-empty set $B_r(y) \cap X$ is finite, for all $y \in \mathbb{E}^n$. Consequently, there is an $x_0 \in X$ such that $\norm{y-x_0} = \min_{x \in X} \norm{y-x}$, hence $y \in V_{x_0}(X)$. We conclude that $\bigcup_{x \in X} V_x(X) = \mathbb{E}^n$.

\noindent Finally, assume that the intersection $V_{x_1}(X) \cap V_{x_2}(X)$ is nonempty. The construction of Voronoi-cells implies that $V_{x_1}(X) \cap V_{x_2}(X) \subset E_{x_1, x_2}$. Assume that $y \in \left( V_{x_1}(X) \cap  E_{x_1, x_2} \right) \setminus \left( V_{x_2}(X) \cap E_{x_1, x_2}) \right)$. Again by construction of Voronoi-cells there must be a point $x_3 \in X$ such that $\norm{x_3-y} < \norm{x_2 - y}$, and $\norm{x_1 - y} \leq \norm{x_3 - y}$. But both inequalities together contradict $\norm{x_1 - y} = \norm{x_2 - y}$. We conclude that
\[  V_{x_1}(X) \cap V_{x_2}(X) = V_{x_1}(X) \cap  E_{x_1, x_2} =  V_{x_2}(X) \cap  E_{x_1, x_2}. \]
Since $V_{x_1}(X) \cap  E_{x_1, x_2}$ is a face of $V_{x_1}(X)$ and $V_{x_1}(X) \cap  E_{x_1, x_2}$ is a face of $V_{x_2}(X)$, the two Voronoi-cells $V_{x_1}(X)$ and $V_{x_2}(X)$ intersect face-to-face.
\end{proof}

\begin{rem} \label{uniform-local-VT-rem}
A careful analysis of the proof shows that for $R \gg 0$ not depending on $x \in X$ the Voronoi-cell $V_x(X)$ is completely determined by the points in $B_R(x)$: Choose the $y_1, \ldots, y_{n+1}$ in a configuration around $x_0$ that is independent of $x_0$, up to isometries, and such that the $y_i$ have the same distance to $x_0$. Then the maximal distance $d$ of points in the Voronoi-cell $V_{x_0}(\{x_0, y_1, \ldots, y_{n+1}\})$ to $x_0$ is also independent of $x_0$. Since the maximal distance $d(x_1, \ldots, x_{n+1})$ of points in the Voronoi-cell $V_{x_0}(\{x_0, x_1, \ldots, x_{n+1}\})$ to $x_0$ depends continuously on $x_1, \ldots, x_{n+1}$ choosing $d \gg 0$ implies $d(x_1, \ldots, x_{n+1}) < \frac{3}{2}d$. The assertion follows.
\end{rem}

\begin{Def}
A tiling $T$ of $\mathbb{E}^n$ is called (\textit{isometrically}) \textit{simple} if there exists a finite set of convex polytopes $t_1, \ldots, t_r$ such that all tiles $t \in T$ are an isometric image $\phi_t(t_i)$ of one of the tiles $t_1, \ldots, t_r$, for an isometry $\phi_t \in \mathrm{Isom}(\mathbb{E}^n)$.

\noindent The tiles $t_1, \ldots, t_r$ are called \textit{prototiles}.
\end{Def}

\noindent In the literature, simple tilings are normally built from translations of prototiles only. Relaxing to isometrically (or \textit{rotationally}) simple tilings can reduce the number of prototiles needed (as in the case of Penrose tilings by rhombs), but there are also rotationally simple tilings whose tiles cannot be obtained by translating a finite number of prototiles. A prominent example is the \textit{pinwheel tiling} constructed by Conway as a substitution tiling of right-angled triangles with side lengths $1$ and $2$ at the right angle (see Figure~\ref{pinwheel-subst-fig}): Radin \cite{Rad94} showed that in such a tiling the triangles point to infinitely many directions.

\begin{figure}[h!  ]
\begin{center}
\begin{tikzpicture}
\draw (0,0) -- (5,0) node [midway, above]{$\sqrt{5}$};
\draw (0,0) -- (4,-2) node [midway, below] {2};
\draw (4,-2) -- (5,0) node [midway, below=5pt] {1};
\draw (2,0) -- (2,-1);
\draw (4,0) -- (4,-2);
\draw (2,-1) -- (4,0);
\draw (2,-1) -- (4,-1);
\end{tikzpicture}
\end{center}
\caption {Conway's triangle decomposition leading to Pinwheel Tiling  }
\label{pinwheel-subst-fig}
\end{figure}
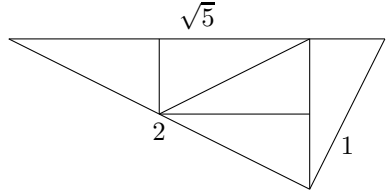

\begin{prop}
Let $X \subset \mathbb{E}^n$ be a Delone set. Then the Voronoi-cell tiling $VT(X)$ of $\mathbb{E}^n$ is simple if for sufficiently large $R \gg 0$ there exist only finitely many point configurations $B_R(x) \cap X$, up to isometries, when $x$ ranges over $X$.
\end{prop}
\begin{proof}
This is a direct consequence of Rem.~\ref{uniform-local-VT-rem}.
\end{proof}

\noindent This proposition justifies the following notion:
\begin{Def}
A Delone set $X \subset \mathbb{E}^n$ is called \textit{simple} if for all $R \gg 0$ there exist only finitely many point configurations $B_R(x) \cap X$ up to isometries when $x$ ranges over $X$.
\end{Def}

\subsection{Tiling Spaces} The set $\mathcal{T}_n$ of all simple tilings of $\mathbb{E}^n$ can be endowed with a number of metrics all inducing the same topology on the set. Tiling spaces will be the closed subsets in this topology that are also invariant under the natural action of the isometry group $\mathrm{Isom}(\mathbb{E}^n)$.

\noindent The main idea to construct a metric on $\mathcal{T}_n$  is to consider tilings to be close if, after applying isometries close to the identity map $\mathbbm{1}_{\mathbb{E}^n}$ on $\mathbb{E}^n$ to the tilings, patches of the tilings covering a large disk coincide. Fixing an origin $O \in \mathbb{E}^n$, the metric may formally be constructed as follows:
\begin{Def} \label{dist-tilings-def}
Let $T, T^\prime$ be two simple tilings of $\mathbb{E}^n$. Define $R(T, T^\prime) > 0$ to be the supremum of all $r>0$ such that there exist $\phi, \psi \in \mathrm{Isom}(\mathbb{E}^n)$ satisfying $d_O(\phi, \mathbbm{1}_{\mathbb{E}^n}), d_O(\psi, \mathbbm{1}_{\mathbb{E}^n}) < \frac{1}{2r}$ and
\[ \left[ \phi(T) \right]_{B_r(O)} = \left[ \psi(T^\prime) \right]_{B_r(O)}. \]
Then the \textit{distance} between $T$ and $T^\prime$ is set to be
\[ d_O(T, T^\prime) := \min \{\ln(3/2), \ln(1+\frac{1}{R(T,T^\prime)})\}. \]
\end{Def}

\noindent The analogous definition of a distance using only translations is standard (see \cite[p.6]{Sad08}). The slight changes to the definition of $d_O(T, T^\prime)$ help to prove the triangle inequality when general isometries are involved (see below), but are in no way the only possible.

\begin{prop}
$(\mathcal{T}_n, d_O)$ is a metric space.
\end{prop}

\begin{proof}
For tilings $T, T^\prime, T^{\prime\prime} \in \mathcal{T}_n$ we have to show that
\begin{enumerate}
\item $d_O(T,T')=d_O(T',T)$,
\item $d_O(T,T') \geq 0$, and $d_O(T,T')=0$ if and only if $T=T'$, and
\item $d_O(T,T'')\leq d_O(T,T')+d_O(T',T'')$.
\end{enumerate}
By definition, $d_O$ is clearly symmetric, positive and $d_O(T,T)=0$.

\noindent Suppose $T\neq T'$. Then there exists tiles $t \in T,t'\in T' $ such that $t\neq t'$ but the interiors $t^\circ  ,t'^\circ  $ intersect. Therefore, for $\epsilon $ small enough we have for all $\phi,\psi\in Isom(\mathbb{E}^n)$ with $d_O ( \phi , \mathbbm{1}_{\mathbb{E}^n}), d_O ( \psi , \mathbbm{1}_{\mathbb{E}^n})<\epsilon $ that \[\phi(t) \neq \psi (t') \text{ and  } (\phi(t))^\circ \cap (\psi(t'))^\circ \neq \emptyset . \]
Choose  $r>\frac{1}{2\epsilon}$ (hence $\epsilon>\frac{1}{2r}$) such that $B_r(O)\cap t$ and $B_r(O)\cap t'$ are both non-empty. The argument above shows that whatever $\phi ,\psi\in Isom(\mathbb{E}^n)$ with $d_O (\phi , \mathbbm{1}_{\mathbb{E}^n} ) , d_O ( \psi , \mathbbm{1}_{\mathbb{E}^n})<\frac{1}{2r}$ we choose, \[  [\phi (T)]_{B_r(O)} \neq [\psi (T')]_{B_r(O)}  \]
The same holds for all $r' \geq r$, hence $R(T,T')\leq r$ and $d_O(T,T')\neq 0$.

\noindent Finally, we have to prove the triangle inequality. If $R(T,T')\leq 2$ or $R(T',T'')\leq  2$ then the triangle inequality follows immediately from the definition of $d_O$. So assume that $R(T,T') >2$ and $R(T',T'')> 2$.

\noindent Then there exist $2< r \leq R(T,T')$ and $\phi , \psi\in Isom(\mathbb{E}^n)$ such that $d_O ( \phi , \mathbbm{1}_{\mathbb{E}^n} ) , d_O ( \psi , \mathbbm{1}_{\mathbb{E}^n})<\frac{1}{2r}$ and
\[ [\phi(T)]_{B_r(O)}=[\psi (T')]_{B_{r}(O)}.\]
Similarly, there exist $2< r'\leq R(T',T'')$ and $\chi , \omega \in Isom(\mathbb{E}^n)$  such that  $d_O(\chi , \mathbbm{1}_{\mathbb{E}^n} ), d_O (\omega , \mathbbm{1}_{\mathbb{E}^n})<\frac{1}{2r'} $ and
 \[ [\chi(T')]_{B_{r'}(O)}=[\omega(T'')]_{B_{r'}(O)}.\]
Choose  $r_0= \frac{rr'}{r+r'}$.  The equality of patches $[\phi (T)]_{B_r(O)}=[\psi (T')]_{B_r(O)}$ implies $[\chi (\phi(T))]_{\chi (B_r(O))}= [\chi(\psi (T'))]_{\chi (B_r(O))}$. Assume $\chi =\tau \cdot \alpha$ where $\tau \in \mathrm{Trans}(\mathbb{E}^n)$, $\alpha \in O(\mathbb{E}^n_O)$. Then we have $\alpha(B_r(O))=B_r(O)$, and consequently $[\chi (\phi(T))]_{\tau (B_r(O))}= [\chi(\psi (T'))]_{\tau (B_r(O))}$. The choice of $r_0$ implies that $B_{r_0}(O) \subset \tau(B_r(O))$ since
\[ \| \tau \| = d_O (\tau, \mathbbm{1}_{\mathbb{E}^n}) \leq d_O (\tau, \mathbbm{1}_{\mathbb{E}^n}) +d_O (\alpha, \mathbbm{1}_{\mathbb{E}^n} )= d_O (\tau \alpha, \mathbbm{1}_{\mathbb{E}^n} )<\frac{1}{2r'} \] and
\[ r-\frac{1}{2r'}=\frac{2rr'-1}{2r'} >r_0, \]
because $r, r^\prime >2$. Hence,
\[ [\chi(\phi(T))]_{B_{r_0}(O)} =  [\chi(\psi(T'))]_{B_{r_0}(O)}. \leqno{(\ast)}\]
$[\chi(T')]_{B_{r'}(O)}=[\omega(T'')]_{B_{r'}(O)}$ implies $[\bar{\psi}\big(\chi(T')\big)]_{\bar{\psi}B_{r'}(O)}=[\bar{\psi}\big(\omega(T'')\big)]_{\bar{\psi}B_{r'}(O)}$, with $\bar{\psi}=\chi \psi \chi^{-1}$. By \ref{d_O-lem} (\ref{met-conj-bd-for}),
\[d_O(\bar{\psi}, \mathbbm{1}_{\mathbb{E}^n}) = d_O(\chi\psi\chi ^{-1}, \mathbbm{1}_{\mathbb{E}^n})\leq d_O (\psi , \mathbbm{1}_{\mathbb{E}^n})\big(1+d_O(\chi,\mathbbm{1}_{\mathbb{E}^n})\big) <\frac{1}{2r}\big(1+\frac{1}{2r'}\big). \]
$r, r' > 2$ implies
\[ r'-\frac{1}{2r}(1+\frac{1}{2r'}) > r_0=\frac{rr'}{r+r'}. \]
Therefore as before, \[ [\chi \psi (T')]_{B_{r_0}(O)}=[\bar{\psi}\big(\chi (T')\big)]_{B_{r_0}(O)}=[\bar{\psi}\big(\omega (T'')\big)]_{B_{r_0}(O)}. \leqno{(\ast\ast)}    \]
Together with $d_O(\chi\phi,\mathbbm{1}_{\mathbb{E}^n})\leq d_O( \chi , \mathbbm{1}_{\mathbb{E}^n})+d_O( \phi, \mathbbm{1}_{\mathbb{E}^n}) \leq \frac{1}{2r'}+\frac{1}{2r}=\frac{1}{2r_0}$ and
\[ d_O( \bar{\psi} \omega, \mathbbm{1}_{\mathbb{E}^n} ) \leq  d_O(\bar{\psi}, \mathbbm{1}_{\mathbb{E}^n})+d_O( \omega, \mathbbm{1}_{\mathbb{E}^n})
   \leq \frac{1}{2r }+ \frac{1}{2r}\frac{1}{2r'} + \frac{1}{2r'} = \frac{1}{2r_0}+\frac{1}{2r} \cdot \frac{1}{2r'}, \]
($\ast$) and ($\ast\ast$) show that
\[ \frac{1}{r_0}+\frac{1}{2r}\frac{1}{r'} \geq \frac{1}{R(T,T'')}. \]
This implies $d(T,T'') \leq d(T,T')+d(T',T'')$ because then
\[ \ln\big(1+\frac{1}{R(T,T'')}\big) \leq \ln\big(1+\frac{1}{r_0} + \frac{1}{2r}\cdot\frac{1}{r'} ) \leq
    \ln (1+\frac{1}{r})+\ln(1+\frac{1}{r'}) \]
by taking $\exp$ of both sides for the second inequality. The triangle inequality is shown.
\end{proof}

\noindent The metric $d_O$ on the space  $\mathcal{T}_n$ of simple tilings heavily depends on the choice of the origin $O \in \mathbb{E}^n$. But the topology induced by the metric is independent of $O$:
\begin{prop} \label{indep-top-met-prop}
The topologies on $\mathcal{T}_n$ induced by the metrics $d_O$ and $d_{O^\prime}$ are equal, for all points $O, O^\prime \in \mathbb{E}^n$.
\end{prop}
\begin{proof}
The underlying reason for the assertion to hold is that by \ref{d_O-lem} (\ref{met-comp-for}), the metrics $d_O, d_{O^\prime}$ on $\mathbb{E}^n$ are comparable, that is, there exists a constant $C > 1$ such that for all $\phi, \psi \in \mathrm{Isom}(\mathbb{E}^n)$,
\[ \frac{1}{C} \cdot d_{O^\prime}(\phi, \psi) \leq d_O(\phi, \psi) \leq C \cdot d_{O^\prime}(\phi, \psi). \]
Furthermore, we use that $B_r(O^\prime) \subset B_{2r}(O)$ for $r > \norm{\overrightarrow{OO^\prime}}_{\mathrm{Eucl}}$, and vice versa.

\noindent In more details, assume that $R_O(T, T^\prime) > 2 \cdot \norm{\overrightarrow{OO^\prime}}_{\mathrm{Eucl}}$ holds for two simple tilings $T, T^\prime$ of $\mathbb{E}^n$. By definition there exist $r > \norm{\overrightarrow{OO^\prime}}_{\mathrm{Eucl}}$ and $\phi, \psi \in \mathrm{Isom}(\mathbb{E}^n)$ such that $d_O(\phi, \mathbbm{1}_{\mathrm{E}^n}), d_O(\psi, \mathbbm{1}_{\mathrm{E}^n}) < \frac{1}{4r}$ and $[\phi(T)]_{B_{2r}(O)} = [\psi(T^\prime]_{B_{2r}(O)}$. Consequently,
\[ d_{O^\prime}(\phi, \mathbbm{1}_{\mathrm{E}^n}), d_{O^\prime}(\psi, \mathbbm{1}_{\mathrm{E}^n}) < \frac{C}{4r} < \frac{C}{2r}\]
and
\[ [\phi(T)]_{B_r(O^\prime)} = [\psi(T^\prime]_{B_r(O^\prime)}. \]
$C>1$ implies $B_{\frac{r}{C}}(O^\prime) \subset B_r(O^\prime)$, and hence $R_{O^\prime}(T, T^\prime) \geq \frac{1}{2C} \cdot R_O(T, T^\prime)$. Reversing the r\^oles of $T$ and $T^\prime$ we conclude $R_{O^\prime}(T, T^\prime) \leq 2C \cdot R_O(T, T^\prime)$ if $R_{O^\prime}(T, T^\prime) > 2 \cdot \norm{\overrightarrow{OO^\prime}}_{\mathrm{Eucl}}$. This shows the comparability of the metrics $d_O$ and $d_{O^\prime}$ for small distances, hence the induced topologies are equal.
\end{proof}

\noindent Next, we show that a tiling moved by an isometry close to the identity map is close to the original tiling, in terms of the metric $d_O$.
\begin{prop} \label{cont-mov-til-prop}
If $T$ is a simple tiling of $\mathbb{E}^n$ and $\phi_k \in \mathrm{Isom}(\mathbb{E}^n)$ are isometries such that $\lim_{k \rightarrow \infty} d_O(\phi_k, \mathbbm{1}_{\mathbb{E}^n}) = 0$ then also
\[ \lim_{k \rightarrow \infty} d_O(\phi_k(T), T) = 0. \]
\end{prop}
\begin{proof}
Write $\phi_k = \tau_k \alpha_k$, where $\tau_k \in \mathrm{Trans}(\mathbb{E}^n)$ and $\alpha_k \in O(\mathbb{E}^n_O)$. By assumption
\[ \norm{\tau}_{\mathrm{Eucl}} = d_O(\tau_k, \mathbbm{1}_{\mathbb{E}^n}) \rightarrow 0\ \mathrm{and}\ \norm{\alpha_k}_{\mathrm{op}} = d_O(\alpha_k, \mathbbm{1}_{\mathbb{E}^n}) \rightarrow 0 \]
when $k \rightarrow \infty$. The triangle inequality implies that
\[ d_O(\phi_k(T), T) = d_O(\tau_k \alpha_k(T), T) \leq d_O(\tau_k(\alpha_k(T)), \alpha_k(T)) + d_O(\alpha_k(T), T). \]
Hence it is enough to show $d_O(\tau_k(\alpha_k(T)), \alpha_k(T)) \stackrel{k \rightarrow \infty}{\longrightarrow} 0$ and $d_O(\alpha_k(T), T) \stackrel{k \rightarrow \infty}{\longrightarrow} 0$.

\noindent To verify the first limit consider only $k$ large enough such that $d_O(\tau_k, \mathbbm{1}_{\mathbb{E}^n}) = \norm{\tau_k} < 2$, and set $\psi_k := - \tau_k/2$, $\chi_k := \tau_k/2$. Then for $r < \frac{1}{d_O(\tau_k, \mathbbm{1}_{\mathbb{E}^n})}$ we have $d_O(\psi_k, \mathbbm{1}_{\mathbb{E}^n}) < \frac{1}{2r}$, $d_O(\chi_k, \mathbbm{1}_{\mathbb{E}^n}) < \frac{1}{2r}$ and $[\psi_k(\tau_k(\alpha_k(T)))]_{B_r(O)} = [\chi_k(\alpha_k(T))]_{B_r(O)}$. By definition this implies $R(\tau_k(\alpha_k(T)), T) \geq \frac{1}{d_O(\tau_k, \mathbbm{1}_{\mathbb{E}^n})}$, hence
\[ d_O(\tau_k(\alpha_k(T)), \alpha_k(T)) \leq \ln(1+ d_O(\tau_k, \mathbbm{1}_{\mathbb{E}^n})) \rightarrow 0 \]
when $k \rightarrow \infty$.

\noindent For the second limit, only consider $k$ large enough such that there is a square root $\beta_k$ of $\alpha_k$ as constructed in Lem.~\ref{sqrt-lem}. If $\gamma_k := \beta_k^{-1}$ is the inverse, Lem.~\ref{sqrt-lem} together with Lem.~\ref{d_O-lem}(\ref{inv-met-for}) implies
\[ \norm{\gamma_k - \mathbbm{1}_{\mathbb{E}^n}}_{\mathrm{op}} =
    \norm{\beta_k - \mathbbm{1}_{\mathbb{E}^n}}_{\mathrm{op}} \leq
    \norm{\alpha_k - \mathbbm{1}_{\mathbb{E}^n}}_{\mathrm{op}}. \]
$\gamma_k \cdot \alpha_k = \beta_k$ implies for any $R>0$ that $[\gamma_k(\alpha_k(T))]_{B_R(O)} = [\beta_k(T)]_{B_R(O)}$. Therefore, $R(\alpha_k(T), T) \rightarrow \infty$ if $\norm{\alpha_k - \mathbbm{1}_{\mathbb{E}^n}}_{\mathrm{op}} \rightarrow 0$ for $k \rightarrow \infty$, hence $d_O(\alpha_k(T), T) \rightarrow 0$, as required.
\end{proof}

\begin{Def}
A \textit{tiling space} $\Omega$ is a set of simple tilings of $\mathbb{E}^n$ made up of the same set of prototiles, finitely many up to isometries, such that
\begin{enumerate}
\item $\Omega$ is closed under isometries, that is, for all isometries $\phi \in \mathrm{Isom}(\mathbb{E}^n)$ and $T \in \Omega$, we also have $\phi(T) \in \Omega$, and
\item $\Omega$ is complete under the metric $d_O$ on the space of all simple tilings of $\mathbb{E}^n$.
\end{enumerate}
\end{Def}

\noindent For every set of prototiles, finite up to isometries, there is a maximal tiling space:
\begin{lem}
The set $\Omega_{\mathcal{P}}$ of all simple tilings of $\mathbb{E}^n$ made up of the same set $\mathcal{P}$ of prototiles, finitely many up to isometries, is a tiling space.
\end{lem}
\begin{proof}
Suppose $(T_m)_{m \in \mathbb{N}}$ is a Cauchy sequence of simple tilings of $\mathbb{E}^n$. Let $s_m$ be the real number such that $\ln(1+s_m) = d_O(T_m, T_{m+1})$. Passing to a subsequence if necessary we can assume that the sequence $(s_m)_{m \in \mathbb{N}}$ is decreasing and $\sum_m s_m < \infty$.

\noindent By definition of $d_O$ there exist $\phi_m, \phi^\prime_m \in \mathbb{E}^n$ such that $d_O(\phi_m, \mathbbm{1}_{\mathbb{E}^n}), d_O(\phi^\prime_m, \mathbbm{1}_{\mathbb{E}^n}) < \frac{s_m}{2}$ and
\[ [\phi_m(T_m)]_{B_{\frac{1}{s_m}(O)}} = [\phi^\prime_m(T_{m+1})]_{B_{\frac{1}{s_m}(O)}}.   \]
Now, $[\phi(T)]_{B_r(O)} = \phi([T]_{B_r(\phi^{-1}(O))})$ for any simple tiling $T$ of $\mathbb{E}^n$, origin $O \in \mathbb{E}^n$, radius $r > 0$ and isometry $\phi \in \mathrm{Isom}(\mathbb{E}^n)$. Hence ($\ast$) is equivalent to
\[ \phi_m([T_m]_{B_{\frac{1}{s_m}}(\phi_m^{-1}(O))}) = \phi^\prime_m([T_{m+1}]_{B_{\frac{1}{s_m}}(\phi_m^{\prime -1}(O))}) \]
or
\[ \phi^{\prime -1}_m \phi_m([T_m]_{B_{\frac{1}{s_m}}(\phi_m^{-1}(O))}) = [T_{m+1}]_{B_{\frac{1}{s_m}}(\phi_m^{\prime -1}(O))}.
   \leqno{(\ast)}\]
Next, the inclusion
\[ B_{\frac{1}{s_m}}(\phi^{\prime -1}_m(O)) \subset B_{\frac{1}{s_{m+1}}}(\phi^{-1}_{m+1}(O)) \leqno{(\ast\ast)}\]
holds because for all $x \in B_{\frac{1}{s_m}}(\phi^{\prime -1}_m(O))$ we have
\begin{eqnarray*}
\norm{x-\phi_{m+1}^{-1}(O)}_{\mathrm{Eucl}} & \leq & \norm{x-\phi_m^{\prime -1}(O)}_{\mathrm{Eucl}} +
          \norm{\phi_m^{\prime -1}(O) - O}_{\mathrm{Eucl}} + \norm{\phi_{m+1}^{-1}(O) - O}_{\mathrm{Eucl}}  \\
 & \leq & \frac{1}{s_m} + \frac{1}{2}s_m + \frac{1}{2}s_{m+1} \leq \frac{1}{s_{m+1}}
\end{eqnarray*}
using the assumptions on $\phi_m^{\prime -1}$ and $\phi_{m+1}$ and by possibly passing to a further subsequence of $(s_m)_{m \in \mathbb{N}}$.

\noindent Define the isometry $\delta_m := \prod_{k=m}^\infty \phi_k^{\prime -1} \phi_k$ where the factors with lower index are to the right. This infinite composition exists because $d_O(\phi_k, \mathbbm{1}_{\mathbb{E}^n}), d_O(\phi^\prime_k, \mathbbm{1}_{\mathbb{E}^n}) < \frac{s_k}{2}$ and $\sum_k s_k < \infty$, hence by Lem.~\ref{d_O-lem}(\ref{inv-met-for}) and (\ref{prod-id-for}) we have $d_O(\delta_{m,M}, \mathbbm{1}_{\mathbb{E}^n}) < 2 \cdot \sum_k s_k$ for $\delta_{m,M} = \prod_{k=m}^M \phi_k^{\prime -1} \phi_k$, and Lem.~\ref{d_O-lem}(\ref{met-prod-for}) implies that $(\delta_{m,M})_{M \in \mathbb{N}}$ is a Cauchy sequence in $\mathrm{Isom}(\mathbb{E}^n)$. This also shows that $\delta_m \rightarrow \mathbbm{1}_{\mathbb{E}^n}$ when $m \rightarrow \infty$. Then
\begin{eqnarray*}
\delta_m([T_m]_{B_{\frac{1}{s_m}}(\phi_m^{-1}(O))}) & = &
\delta_{m+1} \phi_m^{\prime -1} \phi_m([T_m]_{B_{\frac{1}{s_m}}(\phi_m^{-1}(O))}) \\
 & = & \delta_{m+1}([T_{m+1}]_{B_{\frac{1}{s_m}}(\phi_m^{\prime -1}(O))})\ \mathrm{by}\ (\ast) \\
 & \subset & \delta_{m+1}([T_{m+1}]_{B_{\frac{1}{s_{m+1}}}(\phi_{m+1}^{-1}(O))})\ \mathrm{by}\ (\ast\ast).
\end{eqnarray*}
Hence $T := \bigcup_{m=1}^\infty \delta_m([T_m]_{B_{\frac{1}{s_m}}(\phi_m^{-1}(O))})$ is a simple tiling of $\mathbb{E}^n$ made up of the same prototiles as the $T_m$.

\noindent To prove the lemma it is enough to show $d_O(T_m, T) \rightarrow 0$ when $m \rightarrow \infty$. To this purpose choose a sequence $(t_m)_{m \in \mathbb{N}}$ converging to $0$ such that $d_O(\delta_m, \mathbbm{1}_{\mathbb{E}^n}) < \frac{1}{2}t_m$ and $\frac{1}{t_m} \leq \frac{1}{s_m} - \norm{\delta_m\phi_m^{-1}(O)-O}_{\mathrm{Eucl}}$, hence $B_{\frac{1}{t_m}}(O) \subset B_{\frac{1}{s_m}}(\delta_m\phi_m^{-1}(O))$.

\noindent Since by construction, $[T]_{B_{\frac{1}{s_m}}(\delta_m\phi_m^{-1}(O))} = [\delta_m(T_m)]_{B_{\frac{1}{s_m}}(\delta_m\phi_m^{-1}(O))}$ this implies
\[ [T]_{B_{\frac{1}{t_m}}(O)} = [\delta_m(T_m)]_{B_{\frac{1}{t_m}}(O)}, \]
and together with $d_O(\delta_m, \mathbbm{1}_{\mathrm{E}^n}) \leq \frac{1}{2}t_m$ we conclude that
\[ d_O(T_m, T) \leq \ln(1 + t_m) \stackrel{m \rightarrow \infty}{\longrightarrow} 0. \]
\end{proof}

\noindent For every simple tiling there is a minimal tiling space containing it:
\begin{Def}
Let $T$ be a simple tiling of $\mathbb{E}^n$. The \textit{orbit} of $T$ is defined as the set of copies of $T$ moved by isometries,
\[ O(T) := \{\phi(T): \phi \in \mathrm{Isom}(\mathbb{E}^n)\}. \]
The \textit{hull} $\Omega_T$ is the closure of the orbit $O(T)$ in the topological space of all simple tilings.
\end{Def}

\noindent In \cite[p.7p]{Sad08} there are interesting $1$-dimensional examples when the orbit of a simple tiling fails to be closed in the space of all simple tilings.

\begin{thm}
The hull $\Omega_T$ of a simple tiling $T$ of $\mathbb{E}^n$ is compact.
\end{thm}
\begin{proof}
The proof of \cite[Thm.1.1]{Sad08} still works when translations are replaced by isometries.
\end{proof}

\section{Equivalence relations on simple tilings} \label{equiv-tilings-sec}

\subsection{Topologically conjugated tiling spaces}
Let $\gamma$ be an isometry of $\mathbb{E}^n$.

\begin{Def}
For tiling spaces $\Omega$ and $\Omega^\prime$ of simple tilings of $\mathbb{E}^n$, a continuous map $f: \Omega \rightarrow \Omega^\prime$ is called a $\gamma$-\textit{factor map} if for all tilings $T \in \Omega$ and isometries $\phi \in \mathrm{Isom}(\mathbb{E}^n)$ we have
\[ f(\phi(T)) = (\gamma \phi \gamma^{-1})(f(T)). \]
If $f$ is also a homeomorphism, $f$ is called a \textit{topological conjugacy}.
\end{Def}

\noindent Originally, factor maps were defined only requiring that $\phi$ is a translation, and letting $\gamma$ be the identity. However, if we extend the range of $\phi$ to arbitrary isometries then only requiring $f(\phi(T)) = \phi(f(T))$ would imply that some continuous maps between tiling spaces that are factor maps with respect to translations are not any longer factor maps with respect to arbitrary isometries. The reason is that translations commute whereas general isometries do not. The following example and Prop.~\ref{shift-topcon-prop} show how conjugating with an appropriate isometry $\gamma$ resolves this problem.

\begin{exm} \label{top-eq-ex}
Consider the hull $\Omega_T$ of the standard lattice tiling $T$ of the plane $\mathbb{E}^2$ whose tiles are the squares with vertices $(n,m), (n+1,m), (n,m+1), (n+1, m+1)$ for integers $n,m$ (see~Prop.~\ref{cryst-hull-prop} and Ex.~\ref{cryst-til-ex} for a more detailed description of $\Omega_T$). Let $\gamma$ be a translation of $\mathbb{E}^2$ by a vector $v \in \mathbb{R}^2$ and
\[ f: \Omega_T \rightarrow \Omega_T,\ T^\prime \mapsto \gamma(T^\prime) = T^\prime + v \]
a map of $\Omega_T$ onto itself.

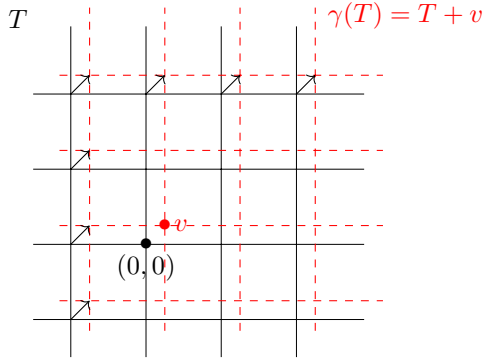
\begin{figure}[h!  ]
\begin{center}
\begin{tikzpicture}[ x=1cm,y=1cm]
\draw[step=1cm,black,very thin] (-1.5,-1.5) grid (2.9,2.9) node at (-1.7,3) {$T$};
\draw[ step=1cm ,red,dashed,shift={(.25,.25)}] (-1.4  ,-1.4 ) grid (2.9,2.9) node at (3.2, 2.8){$\gamma(T)=T+v$};
\draw[step=1 cm,black,thin] (0,0)   node [black] at (0,0) {\textbullet} node[below]{$(0,0)$};
\draw[red] (0.25,.25) node [red] at (.25,.25) {\textbullet} node[right]{$v$};
\draw [->] (-1,0) -- (-.75,0.25);
\draw [->] (-1,1) -- (-.75,1.25);
\draw [->] (-1,2) -- (-.75,2.25);
\draw [->] (-1,-1) -- (-.75,-.75);
\draw [->] (0,2) -- (.25,2.25);
\draw [->] (1,2) -- (1.25,2.25);
\draw [->] (2,2) -- (2.25,2.25);
 \end{tikzpicture}
\end{center}
\caption{ Translated standard lattice tiling in $\mathbb{E}^2$ .}
 \label{fig5}
\end{figure}

\noindent $f$ is a factor map with respect to translations since
\[ f(T'+w)=T'+w+v=f(T')+w. \]
On the other hand $f(\phi(T)) = \phi(T) + v = T + v \neq \phi(T+v) = \phi(f(T))$, if $\phi$ is a counter clockwise rotation by $90^\circ$ around (0,0) and $v$ is suffciently small. But if we allow for an additional isometry on the affine space $\mathbb{E}^2$ (in our case, just $\gamma$), $f$ becomes again a topological conjugacy: For any isometry $\phi$ of $\mathbb{E}^2$,
\[ f(\phi(T)) = \phi(T) + v = \phi(T + v - v) + v = (\gamma\phi\gamma^{-1})(f(T)). \]
\end{exm}

\begin{prop} \label{shift-topcon-prop}
For any isometry $\phi \in \mathrm{Isom}(\mathbb{E}^n)$,
the map $f: \Omega \rightarrow \Omega$ on a space $\Omega$ of tilings of $\mathbb{E}^n$ given by $f(T) := \phi(T)$ for all $T \in \Omega$, is a topological conjugacy.
\end{prop}
\begin{proof}
$f$ is a $\phi$-factor map because for all isometries $\psi \in \mathrm{Isom}(\mathbb{E}^n)$,
\[ f(\psi(T)) = \phi(\psi(T)) = \phi\psi\phi^{-1}(\phi(T)) = \phi\psi\phi^{-1}(f(T)).\]

\noindent To show that $f$ is continuous we use the metrics $d_O$ and $d_{\phi(O)}$ constructed on $\Omega$ as in Def.~\ref{dist-tilings-def} using the origins $O, \phi(O) \in \mathbb{E}^n$. Both metrics define the same topology on $\Omega$ by Prop.~\ref{indep-top-met-prop}. Consequently, the continuity of $f$ follows if for all $\epsilon > 0$ there exists $\delta > 0$ such that
\[ d_O(T, T^\prime) < \delta \Rightarrow d_{\phi(O)}(f(T), f(T^\prime)) < \epsilon. \]
To show this, assume $0 < \epsilon < \ln \frac{3}{2}$ and choose $\delta := \epsilon$. By definition, $d_O(T, T^\prime) < \delta$ means that there is $r > \frac{1}{e^\delta - 1}$ there exist $\psi, \rho \in \mathrm{Isom}(\mathbb{E}^n)$ with $d_O(\psi, \mathbbm{1}_{\mathbb{E}^n}), d_O(\rho, \mathbbm{1}_{\mathbb{E}^n}) < \frac{1}{2r}$ such that
\[ [\psi(T)]_{B_r(O)} = [\rho(T^\prime]_{B_r(O)}. \]
This implies $[\phi\psi(T)]_{B_r(\phi(O))} = [\phi\rho(T^\prime)]_{B_r(\phi(O))}$, hence
\[ [\phi\psi\phi^{-1}(f(T))]_{B_r(\phi(O))} = [\phi\rho\phi^{-1}f(T^\prime))]_{B_r(\phi(O))}. \]
Since $d_{\phi(O)}(\phi\psi\phi^{-1}, \mathbbm{1}_{\mathbb{E}^n}) = d_O(\psi, \mathbbm{1}_{\mathbb{E}^n})$ and $d_{\phi(O)}(\phi\rho\phi^{-1}, \mathbbm{1}_{\mathbb{E}^n}) = d_O(\rho, \mathbbm{1}_{\mathbb{E}^n})$ by Lem.~\ref{d_O-lem}(\ref{conj-met-for}), we obtain $d_{\phi(O)}(f(T), f(T^\prime)) < \ln (1 + \frac{1}{r}) < \delta = \epsilon$.

\noindent Finally, $f$ is a one-to-one map on $\Omega$, and its inverse, given by $f^{-1}(T) := \phi^{-1}(T)$, is a $\phi^{-1}$-factor map by the same reasoning as above. Hence $f$ is a topological conjugacy.
\end{proof}

\subsection{Local derivability} \label{LD-ssec}
The strongest equivalence relation between tiling spaces is \textit{mutual local derivability}. The idea is that a map $f$ between tiling space is already completely determined  locally on the tilings.
\begin{Def} \label{LD-space-Def}
Let $\gamma$ be an isometry of $\mathbb{E}^n$. A $\gamma$-factor map $f: \Omega \rightarrow \Omega^\prime$ between tiling spaces $\Omega, \Omega^\prime$ of simple tilings of $\mathbb{E}^n$ is called \textit{locally derivable} (for short, LD), and we say that $\Omega^\prime$ is $\gamma$-\textit{locally derivable from} $\Omega$ if there exists a real number $R > 0$ such that for all $x \in \mathbb{E}^n$ and tilings $T_1, T_2 \in \Omega$,
\[ [T_1]_{B_R(x)} =  [T_2]_{B_R(x)} \Rightarrow  [f(T_1)]_{\{\gamma(x)\}} = [f(T_2)]_{\{\gamma(x)\}}. \]
Then $R$ is called an LD-radius of $f$.

\noindent If furthermore $f$ is a topological conjugacy and $f^{-1}$ is $\gamma^{-1}$-locally derivable then $\Omega$ and $\Omega^\prime$ are said to be \textit{mutually locally derivable} (for short, MLD).
\end{Def}

\begin{rem}
It is enough to check the $\gamma$-LD property of a $\gamma$-factor map $f: \Omega \rightarrow \Omega^\prime$ at only one point $x_0 \in \mathbb{E}^n$: For a general point $x \in \mathbb{E}^n$, take an isometry $\phi$ mapping $x$ to $\phi(x) = x_0$. Then $[T_1]_{B_R(x)} =  [T_2]_{B_R(x)}$ implies $[\phi(T_1)]_{B_R(x_0)} =  [\phi(T_2)]_{B_R(x_0)}$, hence $[f(\phi(T_1))]_{\{\gamma(x_0)\}} =  [f(\phi(T_2))]_{\{\gamma(x_0)\}}$ because $\phi(T_1), \phi(T_2) \in \Omega$. Since furthermore $f$ is a $\gamma$-factor map, it follows that $[\gamma\phi\gamma^{-1}f(T_1)]_{\{\gamma(x_0)\}} =  [\gamma\phi\gamma^{-1}f(T_2)]_{\{\gamma(x_0)\}}$, or $[f(T_1)]_{\{\gamma\phi^{-1}(x_0)\}} = [f(T_2)]_{\{\gamma\phi^{-1}(x_0)\}}$. But $\phi^{-1}(x_0) = x$.
\end{rem}

\noindent Once again, introducing $\gamma$ makes some natural factor maps MLD:
\begin{prop}
For any isometry $\phi \in \mathrm{Isom}(\mathbb{E}^n)$,
the map $f: \Omega \rightarrow \Omega$ on a space $\Omega$ of tilings of $\mathbb{E}^n$ given by $f(T) := \phi(T)$ for all $T \in \Omega$, is a topological conjugacy making the tiling space $\Omega$ MLD to itself.
\end{prop}
\begin{proof}
We have already shown in Prop.~\ref{shift-topcon-prop} that $f$ is a topological conjugacy.

\noindent $f$ is $\phi$-LD because for all $T_1, T_2 \in \Omega$, equality of patches $[T_1]_{B_R(x)} = [T_2]_{B_R(x)}$ implies $[\phi(T_1)]_{B_R(\phi(x))} = [\phi(T_2)]_{B_R(\phi(x))}$, hence
\[ [f(T_1)]_{B_R(\phi(x))} = [f(T_2)]_{B_R(\phi(x))}. \]
Similarly, the inverse $f^{-1}$ is $\phi^{-1}$-LD
\end{proof}

\noindent To show that being MLD defines an equivalence relation on tiling spaces we need a slight generalization of
Def.~\ref{LD-space-Def}.
\begin{lem} \label{LD-crit-lem}
Let $\Omega, \Omega^\prime$ be two tiling spaces of simple tilings of $\mathbb{E}^n$, and $\gamma$ an isometry of $\mathbb{E}^n$. If $f: \Omega \rightarrow \Omega^\prime$ is a LD $\gamma$-factor map with LD-radius $R$, then for all tilings $T_1, T_2 \in \Omega$, points $x \in \mathbb{E}^n$ and $r > 0$,
\[ [T_1]_{B_{R+r}(x)} =  [T_2]_{B_{R+r}(x)} \Rightarrow  [f(T_1)]_{B_r(\gamma(x))} = [f(T_2)]_{B_r(\gamma(x))}. \]
\end{lem}
\begin{proof}
The ball $B_{R+r}(x)$ is covered by all balls $B_R(x^\prime)$ with $x^\prime \in B_r(x)$, and $[T_1]_{B_{R+r}(x)} =  [T_2]_{B_{R+r}(x)}$ implies $[T_1]_{B_R(x^\prime)} =  [T_2]_{B_R(x^\prime)}$ for all these $x^\prime$.
Local derivability, as defined in Def.~\ref{LD-space-Def}, means that
$[f(T_1)]_{\{\gamma(x^\prime)\}} = [f(T_2)]_{\{\gamma(x^\prime)\}}$, for all $x^\prime \in B_r(x)$. But $\bigcup_{x^\prime \in B_r(x)} \{\gamma(x^\prime)\} = B_r(\gamma(x))$, so $[f(T_1)]_{B_r(\gamma(x))} = [f(T_2)]_{B_r(\gamma(x))}$.
\end{proof}

\begin{lem} \label{MLD-equiv-lem}
Being mutually locally derivable defines an equivalence relation on tiling spaces.
\end{lem}
\begin{proof}
By definition, MLD is reflexive and symmetric on tiling spaces.

\noindent For transitivity, assume that $f: \Omega \rightarrow \Omega^\prime$ and $g: \Omega^\prime \rightarrow \Omega^{\prime\prime}$ are $\gamma_f$-LD resp.\ $\gamma_g$-LD factor maps on tiling spaces $\Omega, \Omega^\prime, \Omega^{\prime\prime}$, with LD-radius $R_f$ resp.\ $R_g$. It is enough to show that $g \circ f$ is a $(\gamma_g \circ \gamma_f)$-LD factor map between $\Omega$ and $\Omega^{\prime\prime}$ with $LD$-radius $R_f + R_g$.

\noindent To this purpose take two tilings $T_1, T_2 \in \Omega$ and set $T_1^\prime := f(T_1)$, $T_2^\prime := f(T_2)$, $T_1^{\prime\prime} := g(T_1^\prime)$ and $T_2^{\prime\prime} := g(T_2^\prime)$. If $[T_1]_{B_{R_f+R_g}(x)} = [T_2]_{B_{R_f+R_g}(x)}$ then Lem.~\ref{LD-crit-lem} implies $[T_1^\prime]_{B_{R_g}(\gamma_f(x))} = [T_2^\prime]_{B_{R_g}(\gamma_f(x))}$, and Def.~\ref{LD-space-Def} shows
\[ [T_1^{\prime\prime}]_{\{\gamma_g(\gamma_f(x))\}} = [T_2^{\prime\prime}]_{\{\gamma_g(\gamma_f(x))\}}, \]
as requested.
\end{proof}

\noindent In the literature (see~\cite{BSJ91}) local derivability was introduced for single tilings, and using only translations of $\mathbb{E}^n$, not more general isometries. Example~\ref{MLD-trans-ex} shows that there are tilings (M)LD with respect to translations, but not with respect to all isometries of $\mathbb{E}^n$, in the sense of Def.~\ref{LD-def} below.

\begin{Def} \label{LD-def}
Let $\gamma$ be an isometry of $\mathbb{E}^n$. A simple tiling $T^\prime$ of $\mathbb{E}^n$ is \textit{$\gamma$-locally derivable from} a simple tiling $T$ of $\mathbb{E}^n$ (for short, $T^\prime$ is \textit{$\gamma$-LD from} $T$) if there exists a real number $R > 0$ such that for all $x \in \mathbb{E}^n$ and $\phi \in \mathrm{Isom}(\mathbb{E}^n)$,
\[ [T]_{B_R(x)} = [\phi(T)]_{B_R(x)}\ \Rightarrow [T^\prime]_{\{\gamma(x)\}} = [\gamma\phi\gamma^{-1}T^\prime]_{\{\gamma(x)\}}. \]
Then $R$ is called an \textit{LD-radius} of $T$ and $T^\prime$.

\noindent If $T^\prime$ is $\gamma$-LD from $T$ and $T$ is $\gamma^{-1}$-LD from $T^\prime$, we say that $T$ and $T^\prime$ are \textit{mutually $\gamma$-locally derivable}, for short, \textit{$\gamma$-MLD}.
\end{Def}

\noindent As for tiling spaces there is a criterion for local derivability generalizing the definition:
\begin{lem} \label{LD-crit2-lem}
Let $T, T^\prime$ be two simple tilings of $\mathbb{E}^n$. If $T^\prime$ is $\gamma$-LD from $T$, with LD-radius $R$, then for all $r > 0$, $x \in \mathbb{E}^n$ and $r > 0$,
\[ [T]_{B_{R+r}(x)} = [\phi(T)]_{B_{R+r}(x)}\ \Rightarrow [T^\prime]_{B_r(\gamma(x))} = [\gamma\phi\gamma^{-1}(T^\prime)]_{B_r(\gamma(x))}.\]
\end{lem}
\begin{proof}
Completely analogous to the proof of Lem.~\ref{LD-crit-lem}.
\end{proof}

\noindent There is a close connection between local derivability of tilings and of tiling spaces.
\begin{prop}
If $f: \Omega \rightarrow \Omega^\prime$ is a topological conjugacy between spaces of tilings of $\mathbb{E}^n$ that makes $\Omega$ and $\Omega^\prime$ MLD, then each tiling $T \in \Omega$ is MLD to the tiling $f(T)$.
\end{prop}
\begin{proof}
Let $T$ be a tiling of $\Omega$ and $\phi$ an isometry of $\mathbb{E}^n$, and assume that $f$ is a $\gamma$-factor map. Then equality of patches $[T]_{B_R(x)} = [\phi(T)]_{B_R(x)}$ implies $[f(T)]_{\{\gamma(x)\}} = [f(\phi(T))]_{\{\gamma(x)\}}$ because $\phi(T)$ is also a tiling in the tiling space $\Omega$. Since $f$ is a $\gamma$-factor map we conclude
\[ [f(T)]_{\{\gamma(x)\}} = [\gamma\phi\gamma^{-1}(f(T))]_{\{\gamma(x)\}}, \]
so $f(T)$ is LD from $T$. Using the inverse factor map $f^{-1}$ in the same way we can also show that $T$ is $LD$ from $f(T)$.
\end{proof}

\begin{prop} \label{LD-til-LD-map-prop}
If $T$ and $T^\prime$ are tilings of $\mathbb{E}^n$ and $\gamma$ is an isometry of $\mathbb{E}^n$ such that $T^\prime$ is $\gamma$-LD from $T$ then there exists a unique locally derivable continuous $\gamma$-factor map
\[ f: \Omega_T \rightarrow \Omega_{T^\prime} \]
between the hull $\Omega_T$ of $T$ and the hull $\Omega_{T^\prime}$ of $T^\prime$ such that $f(T) = T^\prime$.

\noindent In particular, if $T$ and $T^\prime$ are MLD tilings then $\Omega_T$ and $\Omega_{T^\prime}$ are MLD tiling spaces.
\end{prop}
\begin{proof}
Assume that the LD-radius of $T$ and $T^\prime$ is $R$.

\noindent We construct $f$ by setting $f(T) = T^\prime$, extending it to the orbit of $T$ by setting
\[ f(\phi(T)) := \gamma\phi\gamma^{-1}(T^\prime) \]
for all isometries $\phi$ of $\mathbb{E}^n$, and then by continuity to the hull $\Omega_T$ which is the closure of $O(T)$. This construction also shows that any $\gamma$-factor map $f^\prime: \Omega_T \rightarrow \Omega_{T^\prime}$ is uniquely prescribed by the image $f^\prime(T)$ of $T$, hence the uniqueness statement.

\noindent The construction is well-defined because of the next two claims:

\vspace{0.1cm}

\noindent \textit{Claim 1.} $\phi(T) = T$ implies $\gamma\phi\gamma^{-1}(T^\prime) = T^\prime$, for all isometries $\phi$ of $\mathbb{E}^n$.

\noindent \textit{Proof of Claim 1.}
This follows directly from $T^\prime$ being $\gamma$-LD from $T$, as $\phi(T) = T$ implies the equality of patches $[\phi(T)]_{B_R(x)} = [T]_{B_R(x)}$ covering balls of radius $R$ centered in arbitrary points $x \in \mathbb{E}^n$. Consequently, tiles of $\gamma\phi\gamma^{-1}(T^\prime)$ containing $\gamma(x)$ are also tiles of $T^\prime$.
\hfill $\Box$

\vspace{0.1cm}

\noindent Thus we obtain a $\gamma$-factor map $f: O(T) \rightarrow O(T^\prime)$ between the orbits of $T$ and~$T^\prime$.

\vspace{0.1cm}

\noindent \textit{Claim 2.} The map $f: O(T) \rightarrow O(T^\prime)$ is continuous with respect to the topologies induced from the hulls $\Omega_T$ and $\Omega_{T^\prime}$.

\noindent \textit{Proof of Claim 2.}
Let $T_n \rightarrow \bar{T}$ be a convergent sequence in $O(T)$. Choose $\phi_n, \bar{\phi} \in \mathrm{Isom}(\mathbb{E}^n)$ such that $T_n = \phi_n(T)$ and $\bar{T} = \bar{\phi}(T)$. Since the orbit $O(T)$ can be a nowhere closed dense subset of the hull $\Omega_T$ we cannot assume that $\phi_n \rightarrow \bar{\phi}$. Instead we combine the definition of the distance between tilings and that of local derivability.

\noindent $T_n \rightarrow \bar{T}$ tells us that there exists a large $R \gg 0$ and $\psi_n, \bar{\psi}_n$ tending to $\mathrm{id}_{\mathbb{E}^n}$ such that
\[ \left[ \psi_n \phi_n (T) \right]_{B_R(x)} = \left[ \bar{\psi}_n \bar{\phi} (T) \right]_{B_R(x)}. \] 
This implies 
\[ \left[ T \right]_{B_R(\phi_n^{-1} \psi_n^{-1} (x))} = \left[ \phi_n^{-1} \psi_n^{-1} \bar{\psi}_n \bar{\phi} (T) \right]_{B_R(\phi_n^{-1} \psi_n^{-1} (x))}. \]
Since $R$ will eventually be much larger than the the LD-radius of $T$ and $T^\prime$ we may conclude 
\[ \left[ T^\prime \right]_{B_{R^\prime}(\gamma \phi_n^{-1} \psi_n^{-1} (x))}  = 
    \left[ \gamma \phi_n^{-1} \psi_n^{-1} \bar{\psi}_n \bar{\phi} \gamma^{-1} (T^\prime) \right]_{B_{R^\prime}(\gamma  
    \phi_n^{-1} \psi_n^{-1} (x))} \]
for some $R^\prime > R/2$. Hence
\[ \left[ \gamma \phi_n \gamma^{-1}(T^\prime) \right]_{B_{R^\prime}(\gamma \psi_n^{-1} (x))} = 
    \left[ \gamma \psi_n^{-1} \bar{\psi}_n \bar{\phi} \gamma^{-1} (T^\prime) \right]_{B_{R^\prime}(\gamma \psi_n^{-1} (x))} \]
Setting $\bar{\psi}^\prime_n := \gamma \psi_n^{-1} \bar{\psi}_n^{-1} \gamma^{-1}$ and possibly further reducing $R^\prime$ by an arbitrarily small amount we obtain
\[ \left[ \gamma \phi_n \gamma^{-1}(T^\prime) \right]_{B_{R^\prime}(\gamma(x))} = \left[ \bar{\psi}^\prime_n \gamma \bar{\phi} \gamma^{-1}(T^\prime) \right]_{B_{R^\prime}(\gamma(x))} \]
since $d_O(\psi_n^{-1}, \mathrm{id}_{\mathbb{E}^n}) \rightarrow 0$. Then Lem.~\ref{conj-met-for} implies 
\[ \gamma \phi_n \gamma^{-1} (T^\prime) \rightarrow \gamma \bar{\phi} \gamma^{-1} (T^\prime). \]
\hfill $\Box$

\vspace{0.1cm}

\noindent Next we show that $f$ is a $\gamma$-factor map: If $\bar{T} = \lim_{n \rightarrow \infty} \phi_n(T)$ for isometries $\phi_n$ of $\mathbb{E}^n$ then
\begin{eqnarray*}
\gamma\phi\gamma^{-1}f(\bar{T}) & = & \gamma\phi\gamma^{-1}f(\lim_{n \rightarrow \infty} \phi_n(T)) =
     \gamma\phi\gamma^{-1}(\lim_{n \rightarrow \infty} f(\phi_n(T))) \\
 & = & \lim_{n \rightarrow \infty} \gamma\phi\phi_n\gamma^{-1}f(T) = \lim_{n \rightarrow \infty} f(\phi\phi_n(T)) \\
 & = & f(\phi(\lim_{n \rightarrow \infty} \phi_n(T))) = f(\phi(\bar{T})).
\end{eqnarray*}
Here, we again use that multiplication in the group $\mathrm{Isom}(\mathbb{E}^n)$ is continuous.

\noindent Finally we show that $f$ is $\gamma$-LD, in two steps: First, we consider two tilings $T_1 = \phi_1(T), T_2 = \phi_2(T)$ in the orbit of $T$. Equality of patches $[\phi_1(T)]_{B_{R+r}(x)} = [\phi_2(T)]_{B_{R+r}(x)}$, for some $r > 0$, implies
$[T]_{B_{R+r}(\phi_1^{-1}(x))} = [\phi_1^{-1}\phi_2(T)]_{B_{R+r}(\phi_1^{-1}(x))}$.
Since $T$ and $T^\prime$ are $\gamma$-LD with LD-radius $R$ we conclude
\[ [T^\prime]_{B_r(\gamma(\phi_1^{-1}(x)))} =
                                    [\gamma\phi_1^{-1}\phi_2\gamma^{-1}(T^\prime)]_{B_r(\gamma(\phi_1^{-1}(x))))}, \]
using Lem.~\ref{LD-crit2-lem}. Hence $[\gamma\phi_1\gamma^{-1}(T^\prime)]_{B_r(\gamma(x))} = [\gamma\phi_2\gamma^{-1}(T^\prime)]_{B_r(\gamma(x))}$, or equivalently
\[ [f(\phi_1(T))]_{B_r(\gamma(x))} = [f(\phi_2(T))]_{B_r(\gamma(x))}. \]
Next we consider the general case: Let $T_1 = \lim_{n \rightarrow \infty} \phi_n^{(1)}(T)$ and $T_2 = \lim_{n \rightarrow \infty} \phi_n^{(2)}(T)$, and assume that $[T_1]_{B_{R+r}(x)} = [T_2]_{B_{R+r}(x)}$ for a point $x \in \mathbb{E}^n$ and some $r > 0$. Define the distance on $\Omega_T$ using the origin $x$ and the distance on $\Omega_{T^\prime}$ using the origin $\gamma(x)$.

\noindent For a given $R^\prime > R+r$, the limit of $d_x(T_1, \phi_n^{(1)}(T))$ being $0$ implies that for $n \gg 0$ there exist $\phi_n^{(1)\prime}, \phi_n^{(1)\prime\prime} \in \mathrm{Isom}(\mathbb{E}^n)$ such that $d_x(\phi_n^{(1)\prime}, \mathbbm{1}_{\mathbb{E}^n})$ and $d_x(\phi_n^{(1)\prime\prime}, \mathbbm{1}_{\mathbb{E}^n})$ are arbitrarily small and
\[ [\phi_n^{(1)\prime}(T_1)]_{B_{R^\prime}(x)} =  [\phi_n^{(1)\prime\prime}(\phi_n^{(1)}(T))]_{B_{R^\prime}(x)}. \]
Similarly, there exist $\phi_n^{(2)\prime}, \phi_n^{(2)\prime\prime} \in \mathrm{Isom}(\mathbb{E}^n)$ such that $d_x(\phi_n^{(2)\prime}, \mathbbm{1}_{\mathbb{E}^n})$ and $d_x(\phi_n^{(2)\prime\prime}, \mathbbm{1}_{\mathbb{E}^n})$ are arbitrarily small and
\[ [\phi_n^{(2)\prime}(T_2)]_{B_{R^\prime}(x)} =  [\phi_n^{(2)\prime\prime}(\phi_n^{(2)}(T))]_{B_{R^\prime}(x)}. \]
These two equalities of patches imply
\[ [T_1]_{B_{R^\prime}(\phi_n^{(1)\prime -1}(x))} =  [\phi_n^{(1)\prime -1}\phi_n^{(1)\prime\prime}(\phi_n^{(1)}(T))]_{B_{R^\prime}(\phi_n^{(1)\prime -1}(x))} \]
and
\[ [T_2]_{B_{R^\prime}(\phi_n^{(2)\prime -1})(x)} =  [\phi_n^{(2)\prime -1}\phi_n^{(2)\prime\prime}(\phi_n^{(2)}(T))]_{B_{R^\prime}(\phi_n^{(2)\prime -1}(x))}. \]
Since $R^\prime > R+r$ and $\phi_n^{(1)\prime}, \phi_n^{(2)\prime}$ are arbitrarily close to $\mathbbm{1}_{\mathbb{E}^n}$ we have $B_{R+r}(x) \subset B_{R^\prime}(\phi_n^{(1)\prime -1}(x))$ and $B_{R+r}(x) \subset B_{R^\prime}(\phi_n^{(2)\prime -1}(x))$, hence
\[ [T_1]_{B_{R+r}(x)} =  [\phi_n^{(1)\prime -1}\phi_n^{(1)\prime\prime}(\phi_n^{(1)}(T))]_{B_{R+r}(x)} \]
and
\[ [T_2]_{B_{R+r}(x)} =  [\phi_n^{(2)\prime -1}\phi_n^{(2)\prime\prime}(\phi_n^{(2)}(T))]_{B_{R+r}(x)}. \]
Together with our assumptions on $T_1$ and $T_2$ this  implies the existence of $\bar{\phi}^{(1)}_n, \bar{\phi}^{(2)}_n$ close to $\phi^{(1)}_n, \phi^{(2)}_n$ such that
\[ [\bar{\phi}^{(1)}_n(T)]_{B_{R+r}(x)} = [\bar{\phi}^{(2)}_n(T)]_{B_{R+r}(x)}, \]
and furthermore $\bar{\phi}^{(1)}_n(T)$ and $\bar{\phi}^{(2)}_n(T)$ tend to $T_1$ resp.\ $T_2$ when $n \rightarrow \infty$.

\noindent By the first step discussed above we conclude
\[ [f(\bar{\phi}^{(1)}_n(T))]_{B_r(\gamma(x))} = [f(\bar{\phi}^{(2)}_n(T))]_{B_r(\gamma(x))}. \]

\noindent Finally, the next lemma shows that $[f(T_1)]_{\{\gamma(x)\}} = [f(T_2)]_{\{\gamma(x)\}}$, as claimed.
\end{proof}

\begin{lem}
Let $T_n \rightarrow T$ and $S_n \rightarrow S$ be two convergent series of tilings of $\mathbb{E}^n$ such that for a point $x \in \mathbb{E}^n$, some $r > 0$ and all $n \gg 0$, we have $[T_n]_{B_r(x)} = [S_n]_{B_r(x)}$. Then:
\[ [T]_{\{x\}} = [S]_{\{x\}}. \]
\end{lem}
\begin{proof} Let $t \in T$ be a tile containing $x$. Since $T_n \rightarrow T$ there exist tiles $t_n \in T_n$ such that $t_n \rightarrow t$ (with respect to a metric on bounded subsets of $\mathbb{E}^n$ extending the Euclidean metric on points). Since $[T_n]_{B_r(x)} = [S_n]_{B_r(x)}$ the tile $t_n$ also lies in $S_n$, and since $S_n \rightarrow S$ we conclude $t = \lim_{n \rightarrow \infty} t_n \in S$.
\end{proof}

\begin{rem}
If topological conjugacy and mutual local derivability are defined only using translations there exist topologically conjugated but not MLD tiling spaces. For example, Clark and Sadun (\cite{CS06}, \cite[\S 3.6]{Sad08}) show that the standard Penrose tiling and the rational Penrose tiling have topologically equivalent but not MLD hulls constructed using only translations.

\noindent However, the shape changing of the triangle tiles in a standard Penrose tiling to those in the corresponding rational Penrose tiling breaks rotational symmetries, so their hulls are not even topologically conjugated.

\noindent Thus the difference beween topologically conjugated and MLD tiling spaces allowing general isometries seems to be an open question.
\end{rem}

\noindent At least, it is straight forward to prove a necessary metric condition for a $\gamma$-factor map between two tiling spaces to be $\gamma$-LD:
\begin{prop}
Let $\Omega, \Omega^\prime$ be two tiling spaces of simple tilings of $\mathbb{E}^n$. A $\gamma$-locally derivable map $f: \Omega \rightarrow \Omega^\prime$ is Lipschitz-continuous.
\end{prop}
\begin{proof}
Fix a point $x_0 \in \mathbb{E}^n$ and assume that $R > 0$ is an LD radius of $f$. For $T_1, T_2 \in \Omega$, the definition of $d_{x_0}(T_1, T_2) =: \delta < \ln \frac{3}{2}$ implies the existence of $r$ arbitrarily close to $\frac{1}{e^\delta-1}$ and $\phi_1, \phi_2 \in \mathrm{Isom}(\mathbb{E}^n)$ with $d_{x_0}(\phi_i, \mathbbm{1}_{\mathbb{E}^n}) < \frac{1}{2r}$, $i = 1, 2$,
such that
\[ [\phi_1(T_1)]_{B_r(x_0)} = [\phi_2(T_2)]_{B_r(x_0)}. \]
Since $f$ is $\gamma$-LD we know for $r \gg 0$ (that is, $\delta \ll 1$) that
\[ [f(\phi_1(T_1))]_{B_{r-R}(\gamma(x_0))} = [f(\phi_2(T_2))]_{B_{r-R}(\gamma(x_0))}. \]
Since $f(\phi_i(T_i)) = \gamma\phi_i\gamma^{-1}(f(T_i))$ and by Lem.~\ref{d_O-lem}(\ref{conj-met-for})
\[ d_{\gamma(x_0)}(\gamma\phi_i\gamma^{-1}, \mathbbm{1}_{\mathbb{E}^n}) = d_{x_0}(\phi_i, \mathbbm{1}_{\mathbb{E}^n}),\ i = 1, 2, \] we conclude that  for $r \gg 0$
\[ d_{\gamma(x_0)}(f(T_1), f(T_2)) < \ln(1 + \frac{1}{r-R}). \]
In particular, if $r > 2R$ then a short calculation shows that $\ln(1 + \frac{1}{r-R}) < 3 \cdot \ln(1 + \frac{1}{r})$.

\noindent Lipschitz-continuity follows.
\end{proof}

\section{Crystallographic tilings} \label{cryst-tilings-sec}

\subsection{Definitions and first properties} The following notion is natural, but seems not to be present in the literature.
\begin{Def} \label{cryst-tile-def}
A simple tiling $T$ of $\mathbb{E}^n$ is \textit{crystallographic} if its \textit{automorphism group}
\[ \mathrm{Aut}(T) := \{ \phi \in \mathrm{Isom}(\mathbb{E}^n) : \phi(T) = T \} \]
is a crystallographic group.
\end{Def}

\noindent Baake and Grimm at least introduce \textit{crystallographic point sets} in $\mathbb{E}^n$ defined by symmetry groups that are crystallographic \cite[Def.3.1]{BG13} but do not study any equivalences between such point sets.

\begin{prop} \label{cryst-hull-prop}
The hull of a crystallographic tiling $T$ of $\mathbb{E}^n$ is homeomorphic to the topological space $\mathrm{Isom}(\mathbb{E}^n)/\mathrm{Aut}(T)$, and the homeomorphism is equivariant under the natural action of $\mathrm{Isom}(\mathbb{E}^n)$.
\end{prop}
\begin{proof}
By construction, the orbit $O(T)$ of the tiling $T$ is homeomorphic to the quotient space $\mathrm{Isom}(\mathbb{E}^n)/\mathrm{Aut}(T)$, and the homeomorphism is $\mathrm{Isom}(\mathbb{E}^n)$-equivariant. Since $T$ is crystallographic, the quotient space $\mathrm{Isom}(\mathbb{E}^n)/\mathrm{Aut}(T)$ must be compact by the definition of crystallographic groups. In particular, $O(T) = \mathrm{Isom}(\mathbb{E}^n)/\mathrm{Aut}(T)$ is closed and hence the hull of $T$.
\end{proof}

\begin{exm} \label{cryst-til-ex}
The standard lattice tiling $T$ constructed in Ex.~\ref{top-eq-ex} is a first example of a crystallographic tiling, as we show by determining the automorphism group $\mathrm{Aut}(T)$: An automorphism $\phi \in \mathrm{Isom}(\mathbb{E}^2)$ of $T$ is the composition $\phi = \tau \cdot \alpha$ of a translation $\tau \in \mathrm{Trans}(\mathbb{E}^2)$ and an orthogonal map $\alpha \in O(\mathbb{E}^2_O)$ centered in $O = (0,0) \in \mathbb{E}^2$. Since $\alpha$ fixes the origin $O$, the translation $\tau$ is determined by the image of the origin, $\tau(O) = \phi(O) =: (n_O, m_O)$. Since $\phi$ maps vertices of tiles to vertices of tiles, $(n_O, m_O) \in \mathbb{Z}^2$, hence $\tau \in \mathrm{Aut}(T) \cap \mathbb{Z}^2$, and also
\[ \alpha = (-\tau) \cdot \phi \in \mathrm{Aut}(T) \cap GL(\mathbb{E}^2_O, \mathbb{Z}). \]
It is easy to see that there are $8$ orthogonal maps that map the unit square with vertices $(0,0), (0,1), (1,0), (1,1)$ to another square in $\mathbb{E}^2$ having vertices with integer coordinates, namely
\[ \begin{pmatrix} 1 & 0\\ 0 & 1 \end{pmatrix}, \begin{pmatrix}1 & 0\\ 0 & -1 \end{pmatrix},
    \begin{pmatrix} -1 & 0\\ 0 & 1 \end{pmatrix}, \begin{pmatrix} -1 & 0\\ 0 & -1 \end{pmatrix}, \]
\[ \begin{pmatrix} 0 & 1\\ 1 & 0 \end{pmatrix}, \begin{pmatrix}0 & -1\\ 1 & 0 \end{pmatrix},
    \begin{pmatrix} 0 & 1\\ -1 & 0 \end{pmatrix}, \begin{pmatrix}0 & -1\\ -1 & 0 \end{pmatrix}. \]
The group of these matrices is isomorphic to the isometry group of a square $D_4$, so
\[ \mathrm{Aut}(T) \cong \mathbb{Z}^2 \rtimes D_4. \]
Next, we consider the tiling $T^\prime$ of $\mathbb{E}^2$ made up of the rhomb with vertices $(0,0), (1,0), (2,1), (1,1)$ and all its translations by vectors $(n,m) \in \mathbb{Z}^2$.

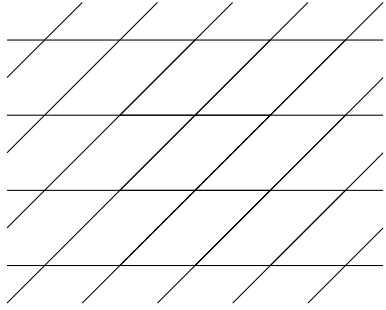
\begin{figure}[h!  ]
\begin{center}
\begin{tikzpicture}[ ]

\draw (1,2) -- (0,2) -- (-1,1) -- (0,1) -- cycle;
\draw (2,2) -- (1,2) -- (0,1) -- (1,1) -- cycle;
\draw (3,2) -- (2,2) -- (1,1) -- (2,1) -- cycle;

\draw (0,1) -- (1,1) -- (0,0) -- (-1,0) -- cycle;
\draw (1,1) -- (2,1) -- (1,0) -- (0,0) -- cycle;
\draw (2,1) -- (3,1) -- (2,0) -- (1,0) -- cycle;

\draw (-0,0) -- (1,0) -- (0,-1) -- (-1,-1) -- cycle;
\draw (1,0) -- (2,0) -- (1,-1) -- (0,-1) -- cycle;
\draw (2,0) -- (3,0) -- (2,-1) -- (1,-1) -- cycle;

\draw(-1,2)--(-1.5,2);
\draw(-1,2)--(-. 5,2.5);
\draw(0,2)--(.5,2.5);
\draw(1,2)--(1.5,2.5);
\draw(2,2)--(2.5,2.5);
\draw(-.5,-1.5)--(3.5,2.5);
\draw(3,2)--(3.5,2);
\draw(3,1)--(3.5,1);
\draw(3,0)--(3.5,0);
\draw(3,-1)--(3.5,-1);

\draw( 2.5,-1.5)--(3.5,- .5);
\draw(1.5,-1.5)--(3.5,.5);
\draw(.5,-1.5)--(3.5,1.5);

\draw(-1,-1)--(-1.5,-1.5);
\draw(-1,0)--(-1.5,-.5);
\draw(-1,1)--(-1.5,0.5);
\draw(-1,2)--(-1.5,1.5);

\draw(0,2)--(-1 , 2 );
\draw(-1,1)--(-1.5,1);
\draw(-1,0)--(-1.5,0);
\draw(-1,-1)--(-1.5,-1);

\draw(2,-1)--(3,-1);

\end{tikzpicture}
\end{center}
\caption{Slanted standard lattice tiling}
  \label{sslt-E2-fig}
\end{figure}

\noindent As before, we may calculate that
\[ \mathrm{Aut}(T^\prime) \cap \mathrm{Trans}(\mathbb{E}^2) = \mathbb{Z}^2,\ \
    \mathrm{Aut}(T^\prime) \cap O(\mathbb{E}^2_O) =
    \left\{ \begin{pmatrix} 1 & 0\\ 0 & 1 \end{pmatrix}, \begin{pmatrix}-1 & 0\\ 0 & -1 \end{pmatrix} \right\}, \]
and consequently, with $D_2$ denoting as usual the "symmetry group of the $2$-gon",
\[ \mathrm{Aut}(T^\prime) \cong \mathbb{Z}^2 \rtimes D_2. \]
In particular, $\mathrm{Aut}(T^\prime)$ is a subgroup of $\mathrm{Aut}(T)$ of index $4$.

\noindent Finally we look at the standard lattice tiling $\bar{T}$ of $\mathbb{E}^2$ contracted by a factor of $2$. That is, the tiles of $\bar{T}$ are squares with side length $\frac{1}{2}$ and vertices with coordinates in $\mathbb{Z} + \frac{1}{2}\cdot \mathbb{Z}$. Obviously,
\[ \mathrm{Aut}(\bar{T}) \cong \mathbb{Z}^2 \rtimes D_4, \]
but as subgroups of $\mathrm{Isom}(\mathbb{E}^2)$, we have $\mathrm{Aut}(T) \subset \mathrm{Aut}(\bar{T})$, and $\mathrm{Aut}(T)$ is a subgroup of index $4$.
\end{exm}

\subsection{Equivalences between crystallographic tilings}

\noindent The main point about introducing equivalence relations between tilings (or tiling spaces) using isometries and not only translations is that these relations allow to distinguish between tilings with different (crystallographic) automorphism groups.

\begin{thm} \label{LD-cryst-thm}
A crystallographic tiling $T^\prime$ of $\mathbb{E}^n$ is $\gamma$-LD from the crystallographic tiling $T$ of $\mathbb{E}^n$ if and only if
\[ \gamma \cdot \mathrm{Aut}(T) \cdot \gamma^{-1} \subset \mathrm{Aut}(T^\prime). \]
\end{thm}
\begin{proof}
First assume that $T^\prime$ is $\gamma$-LD from $T$, and let $R$ be an LD-radius. For any isometry $\rho \in \mathrm{Aut}(T)$ and any tile $t^\prime$ of $T^\prime$ we need to show that $\gamma\rho\gamma^{-1}(t^\prime)$ is another tile of $T^\prime$: Then $\gamma \cdot \rho \cdot \gamma^{-1} \in \mathrm{Aut}(T^\prime)$. So choose an $x \in \mathbb{E}^n$ such that $\gamma(x)$ is in the interior of $\gamma\rho\gamma^{-1}(t^\prime)$. Since $\rho(T) = T$ we have $[T]_{B_R(x)} = [\rho(T)]_{B_R(x)}$. Therefore, $T^\prime$ $\gamma$-LD from $T$ implies that
\[ [T^\prime]_{\{\gamma(x)\}} = [\gamma\rho\gamma^{-1}(T^\prime)]_{\{\gamma(x)\}} = \gamma\rho\gamma^{-1}(t^\prime). \]
In particular, $\gamma\rho\gamma^{-1}(t^\prime)$ is a tile of $T^\prime$.

\noindent Vice versa, assume that $\gamma \cdot \mathrm{Aut}(T) \cdot \gamma^{-1} \subset \mathrm{Aut}(T^\prime)$. Since $\mathrm{Aut}(T)$ is a crystallographic group, $\mathrm{Aut}(T) \cap \mathrm{Trans}(\mathbb{E}^n)$ contains a lattice of full rank by Bieberbach's Theorem~\ref{Bieb-thm}. Hence we can choose a radius $R$ large enough so that for any $x \in \mathbb{E}^n$, we have
\[ \bigcup_{\sigma \in \mathrm{Aut}(T) \cap \mathrm{Trans}(\mathbb{E}^n)} \sigma(B_R(x)) = \mathbb{E}^n. \leqno{(*)} \]
Furthermore, note that for all isometries $\phi \in \mathrm{Isom}(\mathbb{E}^n)$, radii $r>0$ and points $x \in \mathbb{E}^n$,
\[ \phi([T]_{B_r(x)}) = [\phi(T)]_{B_r(\phi(x))}. \leqno{(**)}\]
Now assume that $[T]_{B_R(x)} = [\rho(T)]_{B_R(x)}$ for $x \in \mathbb{E}^n$.

\noindent \textit{Claim}. $\rho \in \mathrm{Aut}(T)$.

\noindent The claim together with $\gamma \cdot \mathrm{Aut}(T) \cdot \gamma^{-1} \subset \mathrm{Aut}(T^\prime)$ implies $\gamma \cdot \rho \cdot \gamma^{-1} \in \mathrm{Aut}(T^\prime)$, hence $[T^\prime]_{\{\gamma(x)\}} = [\gamma\rho\gamma^{-1}(T^\prime)]_{\{\gamma(x)\}}$. Consequently, $T^\prime$ is $\gamma$-LD from $T$.

\noindent \textit{Proof of Claim}. It is enough to show that $\rho(t)$ is a tile of $T$ for all tiles $t \in T$. From ($\ast$) we deduce that there is a $\sigma \in \mathrm{Aut}(T) \cap \mathrm{Trans}(\mathbb{E}^n)$ such that
\[ \sigma(t) \in [T]_{B_R(\rho^{-1}(x))}. \]
The assumption on $\rho$ and ($\ast\ast$) imply that
\[ \rho([T]_{B_R(\rho^{-1}(x))}) = [\rho(T)]_{B_R(x)} = [T]_{B_R(x)}. \]
Therefore there exists a tile $t^\prime \in T$ such that
\[ t^\prime = \rho(\sigma(t)) = (\rho\sigma\rho^{-1}) \cdot \rho(t). \]
In particular, $\rho(t) \in T$ if $\rho\sigma\rho^{-1} \in \mathrm{Aut}(T)$, so it is enough to show that
\[ \rho (\mathrm{Aut}(T) \cap \mathrm{Trans}(\mathbb{E}^n)) \rho^{-1} =
    \mathrm{Aut}(T) \cap \mathrm{Trans}(\mathbb{E}^n). \]
To this purpose choose generators $\sigma_1, \ldots, \sigma_n$ of the lattice of full rank $\mathrm{Aut}(T) \cap \mathrm{Trans}(\mathbb{E}^n)$. Then each $\sigma \in \mathrm{Aut}(T) \cap \mathrm{Trans}(\mathbb{E}^n)$ is a $\mathbb{Z}$-linear combination of these generators, and we have
\[ \rho\sigma\rho^{-1} = \rho(k_1\sigma_1 + \cdots + k_n\sigma_n)\rho^{-1} =
   (\rho\sigma_1\rho^{-1})^{k_1} \cdots (\rho\sigma_n\rho^{-1})^{k_n}. \]
Hence it is enough to show that $\rho\sigma_i\rho^{-1} \in \mathrm{Aut}(T) \cap \mathrm{Trans}(\mathbb{E}^n)$, for $i = 1, \ldots, n$.

\noindent Choose $R$ large enough such that there exists an $R^\prime < R$ for which property ($\ast$) still holds and
\[ R^\prime + \max_{i=1, \ldots, n} \norm{\sigma_i}_{\mathrm{Eucl}} < R. \]
Then the assumption on $\rho$ implies for all $i=1, \ldots, n$ that
\[ \rho\sigma_i\rho^{-1}([T]_{B_{R^\prime}(x)}) = \rho\sigma_i\rho^{-1}([\rho(T)]_{B_{R^\prime}(x)}), \]
and iteratively applying ($\ast\ast$) and using $\sigma_i(T) = T$ yields
\[ \rho\sigma_i\rho^{-1}([\rho(T)]_{B_{R^\prime}(x)}) = [\rho(T)]_{B_{R^\prime}(\rho\sigma_i\rho^{-1}(x))} =
    [T]_{B_{R^\prime}(\rho\sigma_i\rho^{-1}(x))}. \]
The last equality follows because by assumption on $R$ and $R^\prime$,
\[ R^\prime + \norm{\rho\sigma_i\rho^{-1}}_{\mathrm{Eucl}} = R^\prime + \norm{\sigma_i}_{\mathrm{Eucl}} < R, \]
hence $B_{R^\prime}(\rho\sigma_i\rho^{-1}(x)) \subset B_R(x)$, and we can use once again the assumption on $\rho$.

\noindent Consequently, if $t$ is a tile of $[T]_{B_{R^\prime}(x)}$ then $\rho\sigma_i\rho^{-1}(t) \in T$. For tiles $t$ of $T$ not intersecting $B_{R^\prime}(x)$ property ($\ast$) allows us to find a $\tau \in \mathrm{Aut}(T) \cap \mathrm{Trans}(\mathbb{E}^n)$ such that $t \in [T]_{B_{R^\prime}(\tau(x))}$. For such a $\tau$ we have by ($\ast\ast$) that
\[ \rho\sigma_i\rho^{-1}([T]_{B_{R^\prime(\tau(x))}}) = \rho\sigma_i\rho^{-1}(\tau([T]_{B_{R^\prime}(x)})) =
    \tau(\rho\sigma_i\rho^{-1}([T]_{B_{R^\prime}(x)})), \]
since translations commute. Using the equalities above and again ($\ast$) we obtain further
\[ \tau(\rho\sigma_i\rho^{-1}([T]_{B_{R^\prime}(x)})) = \tau([T]_{B_{R^\prime}(\rho\sigma_i\rho^{-1}(x))}) =
    [T]_{B_{R^\prime}(\rho\sigma_i\rho^{-1}(\tau(x))}. \]
Consequently, $\rho\sigma_i\rho^{-1}(t) \in T$, and we can conclude $\rho\sigma_i\rho^{-1} \in \mathrm{Aut}(T)$. \hfill $\Box$

\noindent That finishes the proof.
\end{proof}

\noindent As a direct consequence of this theorem we obtain:
\begin{thm} \label{MLD-cryst-thm}
Two crystallographic tilings of $\mathbb{E}^n$ are $\gamma$-MLD if and only if their automorphism groups are conjugated by the isometry $\gamma$ of $\mathbb{E}^n$. \hfill $\Box$
\end{thm}

\begin{exm} \label{MLD-trans-ex}
The proof of Thm.~\ref{LD-cryst-thm} shows that two crystallographic tilings $T_1, T_2$ of $\mathbb{E}^n$ are MLD with respect to translations if and only if
\[ \mathrm{Aut}(T_1) \cap \mathrm{Trans}(\mathbb{E}^n) = \mathrm{Aut}(T_2) \cap \mathrm{Trans}(\mathbb{E}^n). \]
This is the case for the tilings $T$ and $T^\prime$ constructed in Ex.~\ref{cryst-til-ex}, but not for $T$ and $\bar{T}$. Consequently, $T$ and $T^\prime$ are examples of simple tilings that are MLD with respect to translations, but not with respect to arbitrary isometries.

\noindent This can be visualized using the (counter clockwise) $90^\circ$-rotation $\rho$ of $\mathbb{E}^2$ around the center $x$ of a square in $T$. Since $\rho$ is an automorphism of $T$, we have $[T]_{B_r(x)} = [\rho(T)]_{B_r(x)}$ for any radius $r$, but since the rhomb tiles of $T^\prime$ are not rotated to similarly directed rhombs by $\rho$ we cannot have $[T^\prime]_{\{x\}} = [\rho(T^\prime)]_{\{x\}}$:

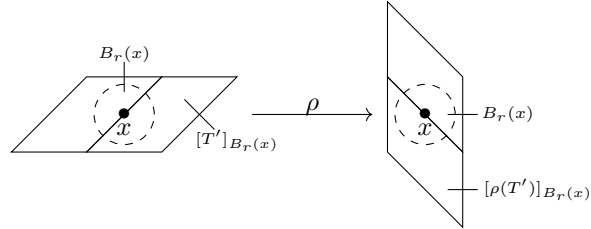
\begin{center}
\begin{tikzpicture}[ ]

\draw (0,0) -- (1,0) -- (2,1) -- (1,1) -- cycle;
\draw (-1,0) -- (0,0) -- (1,1) -- (0,1) -- cycle;
\draw[dashed] (0.5,0.5) circle (0.4cm) node [black] at (0.5,0.5) {\textbullet} node[below] {$x$};
\draw (0.5,1.2) -- (0.5, 0.8) node at (0.5, 1.3) {\tiny $B_r(x)$};
\draw (1.3,0.7) -- (1.7,0.3) node at (2,0.15) {\tiny $[T^\prime]_{B_r(x)}$};

\draw[->] (2.2,0.5) -- (3.8,0.5) node at (3,0.6) {$\rho$};

\draw (4,0) -- (5,-1) -- (5,0) -- (4,1) -- cycle;
\draw (4,1) -- (5,0) -- (5,1) -- (4,2) -- cycle;
\draw[dashed] (4.5,0.5) circle (0.4cm) node [black] at (4.5,0.5) {\textbullet} node[below] {$x$};
\draw (4.8,0.5) -- (5.2, 0.5) node at (5.6, 0.5) {\tiny $B_r(x)$};
\draw (4.8,-0.5) -- (5.2,-0.5) node at (6,-0.5) {\tiny $[\rho(T^\prime)]_{B_r(x)}$};
\end{tikzpicture}
\end{center}

\noindent In a similar way, using a horizontal or vertical translation by $\frac{1}{2}$, we see that $T$ is not LD from $\bar{T}$, not even only with respect to translations.
\end{exm}

\begin{rem}
If $T$ and $T^\prime$ are crystallographic tilings of $\mathbb{E}^n$ such that $\gamma \cdot \mathrm{Aut}(T) \cdot \gamma^{-1} \subset \mathrm{Aut}(T^\prime)$ for some isometry $\gamma$, then by Thm.~\ref{LD-cryst-thm} and Prop.~\ref{LD-til-LD-map-prop} there exists a unique $\gamma$-LD factor map $f: \Omega_T \rightarrow \Omega_{T^\prime}$ between the hulls of $T$ and~$T^\prime$.

\noindent By Prop.~\ref{cryst-hull-prop} these hulls are described as
\[ \Omega_T \cong \mathrm{Isom}(\mathbb{E}^n)/\mathrm{Aut(T)},\
    \Omega_{T^\prime}\cong \mathrm{Isom}(\mathbb{E}^n)/\mathrm{Aut(T^\prime)}, \]
with $T$ and $T^\prime$ mapped to the residue classes of $\mathbbm{1}_{\mathbb{E}^n}$. Hence the hulls are the same as the orbits of $T$ resp. $T^\prime$. The construction in Prop.~\ref{LD-til-LD-map-prop} shows that $f$ can then be identified with the map
\[ \mathrm{Isom}(\mathbb{E}^n)/\mathrm{Aut}(T) \rightarrow \mathrm{Isom}(\mathbb{E}^n)/\mathrm{Aut}(T^\prime)\]
given by conjugation with $\gamma$. In particular, $f$ is a surjective $d\mathrm{-to-}1$ map where $d := [\mathrm{Aut}(T^\prime) : \gamma \cdot \mathrm{Aut}(T) \cdot \gamma^{-1}]$ is the subgroup index.

\noindent Vice versa, if $f: \Omega_T \rightarrow \Omega_{T^\prime}$ is a $\gamma$-factor map between the hulls of crystallographic tilings $T, T^\prime$ of $\mathbb{E}^n$ with $f(T) = T^\prime$ then $\phi \in \mathrm{Aut}(T)$ implies
\[ T^\prime = f(\phi(T)) = \gamma \phi \gamma^{-1} (f(T)) = \gamma \phi \gamma^{-1} (T^\prime), \]
hence $\gamma \phi \gamma^{-1} \mathrm{Aut}(T^\prime)$. Consequently, $\gamma \cdot \mathrm{Aut}(T) \cdot \gamma^{-1} \subset \mathrm{Aut}(T^\prime)$, and by Thm.~\ref{LD-cryst-thm}, $T^\prime$ is $\gamma$-LD from $T$. So in particular, MLD equivalence is the same as topological conjugacy for crystallographic tilings.

\noindent Considering the tilings $T, T^\prime$ constructed in Ex.~\ref{cryst-til-ex}, $\mathrm{Aut}(T^\prime)$ is a subgroup of $\mathrm{Aut}(T)$ of index $4$. The following figure indicates one tile of each of the four tilings in $\Omega_{T^\prime}$ mapped to the same standard lattice tiling in $\Omega_{T}$.

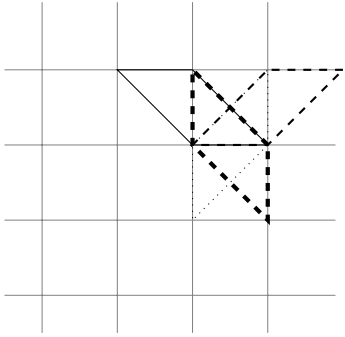
\begin{center}
\begin{tikzpicture}[ ]
\draw[step=1cm,gray,very thin] (-1.5,-1.5) grid (2.9,2.9);
\draw[dashed,thick] (3,2) -- (2,2) -- (1,1) -- (2,1) -- cycle;
\draw  (1,1) -- (2,1) --(1,2)--(0,2)-- cycle;
\draw[dashed,ultra thick](2,1) -- (1,2) -- (1,1) -- (2,0) -- cycle;
\draw[dotted] (2,1) -- (2,2) -- (1,1) -- (1,0) -- cycle;

 \end{tikzpicture}
\end{center}
\end{rem}

\subsection{Tilings with prescribed crystallographic automorphism group}

\begin{thm} \label{cryst-group-tiling-thm}
For a given crystallographic group $\Gamma \subset \mathrm{Isom}(\mathbb{E}^n)$ there exists a simple tiling $T$ of $\mathbb{E}^n$ such that $\mathrm{Aut}(T) = \Gamma$.
\end{thm}

\noindent For the proof of this theorem we need the following propositions:
\begin{prop} \label{stab-prop}
For almost all $x\in \mathbb{E}^n$, the stabilizer $\mathrm{Stab}_\Gamma(x)$ of a crystallographic group $\Gamma \subset Isom(\mathbb{E}^n)$ is trivial.
\end{prop}
\begin{proof}
By Prop.~\ref{symmorph-prop} there exists a crystallographic group $\Gamma^\ast \subset \mathrm{Isom}(\mathrm{E}^n)$ symmorphic with respect to a point $P \in \mathrm{E}^n$ and containing $\Gamma$. It is enough to show that $\mathrm{Stab}_{\Gamma^\ast}(x) = \{\mathbbm{1}_{\mathrm{E}^n}\}$.

\noindent Since $\Gamma^\ast$ is symmorphic with respect to $P$ we have that
\[ \Gamma^\ast = (\Gamma^\ast \cap \mathrm{Trans}(\mathbb{E}^n)) \cdot G \subset \mathrm{Isom}(\mathrm{E}^n) \]
for a finite group $G \subset O(\mathbb{E}^n_P)$. An isometry $\rho = \tau \cdot \alpha$, with $\tau \in \Gamma^\ast \cap \mathrm{Trans}(\mathbb{E}^n)$ and $\alpha \in G$, is in $\mathrm{Stab}_{\Gamma^\ast}(x)$ if and only if $\alpha(x) = x - \tau$, that is
\[ (\alpha - \mathbbm{1}_{\mathrm{E}^n})(x) \in P + (\Gamma^\ast \cap \mathrm{Trans}(\mathbb{E}^n)). \]
But the subspace $\mathrm{Im}(\alpha-\mathbbm{1}_{\mathrm{E}^n}) \subset \mathbb{E}^n_P$ intersects the lattice $P + (\Gamma^\ast \cap \mathrm{Trans}(\mathbb{E}^n))$ also in a lattice. Consequently, $\rho = \tau \cdot \alpha \not\in \mathrm{Stab}_{\Gamma^\ast}(x)$ for almost all $x \in \mathbb{E}^n$ as long as $\mathrm{Im}(\alpha-\mathbbm{1}_{\mathrm{E}^n}) \neq (0)$, that is, $\alpha \neq \mathbbm{1}_{\mathrm{E}^n}$. If  $\alpha = \mathbbm{1}_{\mathrm{E}^n}$ then $\tau \cdot \mathbbm{1}_{\mathrm{E}^n} = \tau \not\in \mathrm{Stab}_{\Gamma^\ast}(x)$ for all $x \in \mathbb{E}^n$ as long as $\tau \neq 0$.

\noindent Since the $\alpha$ vary over the finite group $G$ this implies the claim.
\end{proof}

\begin{prop} \label{VT-cryst-prop}
Let $\Gamma \subset Isom(\mathbb{E}^n)$ be a crystallographic group. Then
for every point $x\in \mathbb{E}^n$, the orbit $\Gamma \cdot x$ is a Delone set. Furthermore, the associated Voronoi-cell tiling $VT(\Gamma \cdot x)$   is a simple tiling.
\end{prop}
\begin{proof}
Since $\Gamma$ is crystallographic $\Gamma \cap \mathrm{Trans}(\mathbb{E}^n)$ is a lattice of full rank, by Bieberbach's Theorem~\ref{Bieb-thm}. Hence $\Gamma \cdot x \supset (\Gamma \cap \mathrm{Trans}(\mathbb{E}^n)) \cdot x$ is certainly relatively dense.

\noindent To prove uniform discreteness we pass again to a crystallographic group $\Gamma^\ast$ containing $\Gamma$ and being symmorphic with respect to a point $P \in \mathbb{E}^n$, as in the proof of Prop.~\ref{stab-prop}. So let $G \subset O(\mathbb{E}^n_P)$ be the finite group such that $\Gamma^\ast = (\Gamma^\ast \cap \mathrm{Trans}(\mathbb{E}^n)) \cdot G$. Then $\Gamma^\ast \cdot x = (\Gamma^\ast \cap \mathrm{Trans}(\mathbb{E}^n)) \cdot (G \cdot x)$ is uniformly discrete, as $G \cdot x$ is finite and $\Gamma^\ast \cap \mathrm{Trans}(\mathbb{E}^n) \subset \mathrm{Trans}(\mathbb{E}^n)$ is a discrete subgroup.

\noindent So the subset $\Gamma \cdot x \subset \Gamma^\ast \cdot x$ is also uniformly discrete, and $\Gamma \cdot x$ is a Delone set. Then Prop.~\ref{VT-prop} implies that the Voronoi-cell tiling $VT(\Gamma \cdot x)$ is a tiling. It has only a finite number of tiles up to isometries because $V_{\gamma (p)}=\gamma (V_p)$:
\[  \norm{\gamma (p)-\gamma (q)} =\norm{p-q} .\]
\end{proof}

\begin{thm}
For every crystallographic group $\Gamma \subset Isom(\mathbb{E}^n)$ there exists a simple tiling $T$ with $Aut(T)=\Gamma$.
\end{thm}
\begin{proof}
We can construct a simple tiling $T$ with $Aut(T)=\Gamma$ as follows:
\begin{enumerate}
\item Choose $x \in \mathbb{E}^n$ such that $Stab_\Gamma(\lbrace x\rbrace )=\lbrace \mathbbm{1}_{\mathbb{E}^n}\rbrace$. Such an $x$ must exist by Prop.~\ref{stab-prop}.
\item Construct the Voronoi-cell tiling $VT(\Gamma\cdot x)$ of the orbit $\Gamma \cdot x$. This is a simple tiling by Prop.~\ref{VT-cryst-prop}, and for every tile $t \in VT(\Gamma\cdot x)$ there exists exactly one $\gamma \in \Gamma$ such that $t = \gamma(V_x(\Gamma \cdot x))$.
\item Let $ t_x \in VT(\Gamma \cdot x) $ be the Voronoi-cell containing $x$, and choose a sufficiently general point $y \in t_x$. Subdivide each tile $\gamma \cdot t_x$ of $VT(\Gamma \cdot x)$, $\gamma \in \Gamma$ by cones having each facet of $\gamma \cdot t_x$ as the basis and the point $\gamma \cdot y$ as the vertex.
\end{enumerate}
This subdivision tiling $T_\Gamma$ is simple and will have automorphism group $\Gamma$, because the possibly existing  additional automorphisms of $VT(\Gamma \cdot x)$ not in $\Gamma$ do not map the subdivision cones onto each other.

\noindent In more details, choose $y$ such that the distances of $y$ to the vertices of $t_x$ are mutually distinct, and are also different from all the lengths of edges of $t_x$. This is possible if we choose $y$ away from a finite number of spheres around the vertices of $t_x$  with radii equal to the edge lengthes  of $t_x$ and also away from the finite number of hyperplanes reflecting one vertex of $t_x$ to another.

\noindent Obviously, $\Gamma \subset Aut(T_\Gamma)$. On the other hand, let $\delta \in Isom(\mathbb{E}^n)$ be an automorphism of $T_\Gamma$.  Let $C_1 \cup C_2\cup...\cup C_r=t_x$ be the subdivision of $t_x$ into cones $C_i$. Note that $\delta (C_1)$ must be one of the subdivision cones in a tile $\gamma \cdot t_x \in VT(\Gamma \cdot x)$. Since the lengths of the edges to the vertex of the cone are all different by construction, $C_1':=\delta (C_1)=\gamma(C_1)$, and the vertex and the edges of $C_1'$ must be mapped to the vertex of $C_1'$ and the same edges by $\delta $ and $\gamma$. Since vertices of $C_1$ span all of $\mathbb{E}^n$, $\delta $ and $\gamma$ are uniquely determined by the images of these vertices, as affine transformations of $\mathbb{E}^n$, hence are equal.
\end{proof}

\begin{exm}
The subdivision of the Voronoi-cell tiling is necessary to kill additional automorphisms, as the example of the standard lattice group $\Gamma=\mathbb{Z} \cdot e_1 \oplus \mathbb{Z}\cdot e_2 \subset \mathrm{Isom}(\mathbb{E}^n)$ shows, with $e_1$ and $e_2$ orthogonal standard basis vectors in $\mathbb{R}^2$:

\noindent Since all the elements of $\Gamma$ are translation, every point $x\in \mathbb{E}^2$ has trivial stabilizer in $\Gamma$. The Voronoi-cell tiling of $\Gamma \cdot x$ is the standard lattice tiling $T$ consisting of tiles  which are squares with vertices of the form $(n,m),(n+1,m),(n,m+1),(n+1,m+1) ;\;\; n,m\in  \mathbb{Z}$, as described in Fig.~\ref{stand-lattice-fig}. And in Ex.~\ref{cryst-til-ex}, we have calculated that
\[ Aut(T)= \Gamma \rtimes D_4. \]

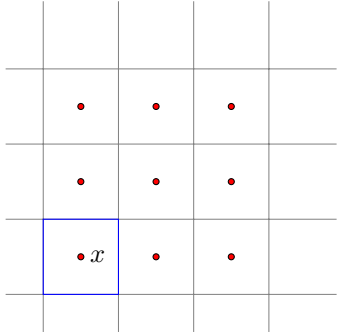
\begin{figure}[h!  ]
\begin{center}
\begin{tikzpicture}[ x=1cm,y=1cm]
\draw[step=1cm,gray,very thin] (-1.5,-1.5) grid (2.9,2.9);
 \draw[fill=red] (1.5,1.5) circle(.04)   ;
\draw[fill=red] (.5,1.5) circle(.04)   ;
\draw[fill=red] (-.5,1.5) circle(.04)   ;
\draw[fill=red] (1.5,.5) circle(.04)   ;
\draw[fill=red] (1.5,-.5) circle(.04)   ;
\draw[fill=red] (.5,.5) circle(.04)   ;
\draw[fill=red] (.5,-.5) circle(.04)   ;
\draw[fill=red] (-.5,.5) circle(.04)   ;
\draw[fill=red] (-.5,-.5) circle(.04) node[right]{$x$}  ;
 \draw[blue] (0,0) -- (-1,0) -- (-1,-1) -- (0,-1) -- cycle;

 \end{tikzpicture}
\end{center}
\caption{Voronoi-cell tiling of $\mathbb{Z}\cdot e_1\oplus \mathbb{Z}\cdot e_2$}
 \label{stand-lattice-fig}
\end{figure}

\noindent After the subdivision with $y$ sufficiently general we obtain a tiling with automorphism group $\mathbb{Z}\cdot e_1\oplus \mathbb{Z}\cdot e_2$.

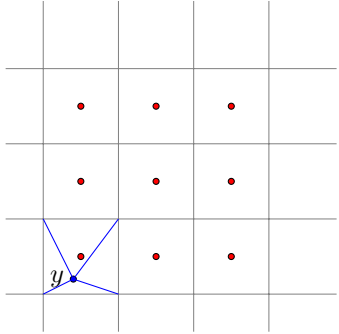
\begin{figure}[h!  ]
\begin{center}
\begin{tikzpicture}[ x=1cm,y=1cm]
\draw[step=1cm,gray,very thin] (-1.5,-1.5) grid (2.9,2.9);
 \draw[fill=red] (1.5,1.5) circle(.04)   ;
\draw[fill=red] (.5,1.5) circle(.04)   ;
\draw[fill=red] (-.5,1.5) circle(.04)   ;
\draw[fill=red] (1.5,.5) circle(.04)   ;
\draw[fill=red] (1.5,-.5) circle(.04)   ;
\draw[fill=red] (.5,.5) circle(.04)   ;
\draw[fill=red] (.5,-.5) circle(.04)   ;
\draw[fill=red] (-.5,.5) circle(.04)   ;
\draw[fill=red] (-.5,-.5) circle(.04)   ;
\draw[fill=blue] (-.6,-.8) circle(.04) node[left]{$y$};
 \draw[blue] (-.6,-.8) -- (-1,0);
 \draw[blue] (-.6,-.8) -- (0,0);
 \draw[blue] (-.6,-.8) -- (0,-1);
 \draw[blue] (-.6,-.8) -- (-1,-1);

 \end{tikzpicture}
\end{center}
\caption{Tiling of $\mathbb{E}^2$ with automorphism group $\mathbb{Z}\cdot e_1\oplus \mathbb{Z}\cdot e_2$.}
 \label{lattice-aut-fig}
\end{figure}
\end{exm}

\bibliographystyle{alpha}


\def\cprime{$'$}

\end{document}